\PassOptionsToPackage{dvipdfmx}{graphicx}
\PassOptionsToPackage{dvipdfmx}{hyperref}
\PassOptionsToClass{a4j,12pt}{article}

\documentclass{article} 

\usepackage{amsmath,amsthm,amssymb}
\usepackage{graphicx, color}
\usepackage{geometry}
\date{}
\ifx\kanjiskip\undefined\else
  \usepackage{atbegshi}
  \ifx\ucs\undefined
    \ifnum 42146=\euc"A4A2
      \AtBeginShipoutFirst{}
    \else
      \AtBeginShipoutFirst{}
    \fi
  \else
    \AtBeginShipoutFirst{}
  \fi
  \usepackage{hyperref}
\fi
\usepackage{url}
\usepackage{cancel}

\theoremstyle{definition}
\newtheorem{theo}{Theorem}[section]
\newtheorem{defi}[theo]{Definition}
\newtheorem{lemm}[theo]{Lemma}
\newtheorem{prop}[theo]{Proposition}

\newtheorem{cor}[theo]{Corollary}
\newtheorem{rem}[theo]{Remark}
\newtheorem{assum}{Assumption}

\newcommand{\relmiddle}[1]{\mathrel{}\middle#1\mathrel{}}
 
\newcommand{\1}{\mbox{1}\hspace{-0.25em}\mbox{l}}
\newcommand{\rsi}{]\hspace{-0.35mm}]}
\newcommand{\lsi}{[\hspace{-0.35mm}[}

\newcommand{\filtprobsp}{(\Omega,\mathcal{F},(\mathcal{F}_t)_{t\in[0,T]},\bold{P})}
\newcommand{\cadlag}{\text{c\`{a}dl\`{a}g}}

\makeatletter
 
  \@addtoreset{equation}{subsection}
\makeatother
\def\widebar{\accentset{{\cc@style\underline{\mskip10mu}}}}
\numberwithin{equation}{section}

\def\rnum#1{\expandafter{\romannumeral #1}} 
\def\Rnum#1{\uppercase\expandafter{\romannumeral #1}}

\title{BSDEs driven by cylindrical martingales with application to approximate hedging in bond markets}

\author{Yushi Hamaguchi\thanks{Department of Mathematics, Kyoto University, Kyoto 606--8502, Japan, hamaguchi@math.kyoto-u.ac.jp}}

\begin{document}
\maketitle


\begin{abstract}
We consider Lipschitz-type backward stochastic differential equations (BSDEs) driven by cylindrical martingales on the space of continuous functions. We show the existence and uniqueness of the solution of such infinite-dimensional BSDEs and prove that the sequence of solutions of corresponding finite-dimensional BSDEs approximates the original solution. We also consider the hedging problem in bond markets and prove that, for an approximately attainable contingent claim, the sequence of locally risk-minimizing strategies based on small markets converges to the generalized hedging strategy.
\end{abstract}


\section{Introduction}
\paragraph{}
In mathematical finance, backward stochastic differential equations (BSDEs) have been studied and applied to the theory of option hedging and portfolio optimization problems in stock markets where a finite-number of assets are traded. El Karoui, Peng, and Quenez \cite{a_ElKaroui-_97} studied hedging problems, recursive utilities, and control problems in terms of finite-dimensional BSDEs. On the other hand, infinite-dimensional (forward) SDEs have also been extensively studied and applied to mathematical finance; see Da Prato and Zabczyk \cite{b_DaPrato-Zabczyk_14} for a summary of infinite-dimensional SDEs and Carmona and Tehranchi \cite{b_Carmona-Tehranchi_06} for applications to bond markets. In this paper, motivated by hedging problems in bond markets, we study some infinite-dimensional BSDEs.\par
As is mentioned by Bj\"{o}rk et al.\ \cite{a_Bjork-_97}, in the continuous-time bond market, unlike in the stock market, there exists a continuum of tradable assets (zero-coupon bonds parametrized by their maturities), and the time evolution of the price curve is described by an infinite-dimensional stochastic process. To describe the portfolio theory in this model, we have to consider trading strategies in which possibly a continuum of zero-coupon bonds of each maturity can contribute. Hence, in the bond market, stochastic integrals with respect to infinite-dimensional (semi)martingales naturally arise and the term ``trading strategy'' has to be generalized in the infinite-dimensional setting.\par
A theory of stochastic integration with respect to cylindrical martingales that is suitable for this purpose has been studied by De Donno and Pratelli \cite{a_DeDonno-Pratelli_04} and Mikulevicius and Rozovskii \cite{ic_Mikulevicius-Rozovskii_98,ic_Mikulevicius-Rozovskii_99}. They introduced a space of generalized integrands and defined generalized stochastic integrals with respect to infinite-dimensional martingales. More generally, De Donno and Pratelli \cite{ic_DeDonno-Pratelli_06,a_DeDonno-Pratelli_05} studied stochastic integration with respect to infinite-dimensional semimartingales (see also De Donno \cite{a_DeDonno_07}). In the context of market models, generalized integrands are regarded as generalized trading strategies, which do not necessarily represent realistic trading strategies. Generalized strategies are defined as limits of realistic trading strategies in bond markets such as simple trading strategies and measure-valued strategies. From a viewpoint of applications to mathematical finance, it is important to construct a ``reasonable'' approximate sequence for generalized integrands.\par
The main purpose of this paper is to contribute the above mentioned problem. We study Lipschitz-type BSDEs driven by cylindrical martingales on the space of continuous functions. We shall show the existence and uniqueness of the solution of such infinite-dimensional BSDEs and prove that the sequence of solutions of corresponding finite-dimensional BSDEs approximates the original solution. We may say that this approximation justifies the formulation of bond markets in terms of infinite-dimensional BSDEs.\par
The hedging problem in bond markets is also discussed as an application. We consider bond market models which satisfy the so-called ``structure condition'' as in the finite-dimensional setting. We introduce concepts of generalized trading strategies and approximately attainable contingent claims and construct submarkets consisting of a finite number of zero-coupon bonds, which we will call small markets. We consider locally risk-minimizing strategies based on small markets, which were defined by Schweizer \cite{a_Schweizer_91,ic_Schweizer_01,ic_Schweizer_08} and studied by Buckdahn \cite{ic_Buckdahn_95} in terms of BSDEs driven by (one-dimensional) martingales. We show that, for approximately attainable contingent claims in bond markets, the generalized hedging strategy is approximated by a sequence of locally risk-minimizing strategies based on small markets. This result further develops the approaches to approximate value-functions in infinite-dimensional markets by the corresponding finite-dimensional value-functions, as studied by De Donno, Guasoni, and Pratelli \cite{a_DeDonno-_05} and Campi \cite{a_Campi_09}.\par
This paper is organized as follows. In Section \ref{section: cylindrical mart}, we review stochastic integration with respect to cylindrical martingales on the space of continuous functions as developed in \cite{a_DeDonno-Pratelli_04} and \cite{ic_Mikulevicius-Rozovskii_98,ic_Mikulevicius-Rozovskii_99}. Section \ref{section: BSDE} is devoted to the study of BSDEs driven by cylindrical martingales; we show the existence, uniqueness, and an approximation result. In Section \ref{section: application}, we study an application to the hedging problem in bond markets and show an approximation result for generalized strategies.


\section{Stochastic integration with respect to cylindrical martingales}\label{section: cylindrical mart}
\paragraph{}
In this section, we shall recall briefly the theory of infinite-dimensional stochastic integrations with respect to cylindrical martingales, following Mikulevicius and Rozovskii \cite{ic_Mikulevicius-Rozovskii_98,ic_Mikulevicius-Rozovskii_99} and De Donno and Pratelli \cite{a_DeDonno-Pratelli_04}.\par
Let $\filtprobsp$ be a filtered probability space satisfying the usual conditions and $\mathcal{P}$ be the predictable $\sigma$-field on $\Omega\times[0,T]$. Throughout the paper, we assume that $T\in(0,\infty)$ is a constant. We denote by $\mathcal{H}^2(\bold{P})$ the space of all right-continuous square-integrable martingales. Let $X$ be a compact metric space and $\mathcal{B}(X)$ be the Borel $\sigma$-field on $X$. In our applications to bond markets, $X$ will be considered as a compact interval $[0,T^*]$ representing the set of maturities of zero-coupon bonds. Let $\mathcal{C}=\mathcal{C}(X)$ be the space of all continuous functions on $X$ with the topology of uniform convergence and $\mathcal{M}=\mathcal{M}(X)$ be its topological dual, i.e., the space of Radon measures on $X$. It is well-known that $\mathcal{M}$ is separable with respect to the weak$^*$-topology $\sigma(\mathcal{M},\mathcal{C})$. We denote the canonical pairing by $\langle\cdot,\cdot\rangle_{\mathcal{M},\mathcal{C}}$. Denote by $\mathcal{K}^+(X)$ the space of all symmetric and nonnegative-definite functions on $X\times X$, i.e., the functions $F\colon X\times X\to\mathbb{R}$ such that $F(x,y)=F(y,x)$ for all $x,y\in X$ and $\sum^d_{i,j=1}F(x_i,x_j)c_ic_j\geq0$ for all $d\in\mathbb{N}$, $x_1,\dots,x_d\in X$, and $c_1,\dots,c_d\in\mathbb{R}$. The set of all $F\in\mathcal{K}^+(X)$ which is continuous on $X\times X$ is denoted by $\mathcal{K}^+_{\text{c}}(X)$.\par
Consider a family of square-integrable martingales $\mathbb{M}=((M^x_t)_{t\in[0,T]})_{x\in X}$, that is, for all $x\in X$, $M^x\in\mathcal{H}^2(\mathbb{P})$. We impose the following assumption.
\begin{assum}\label{assumption: cylindrical mart}
There exist a strictly increasing and bounded predictable process $A$ and a $\mathcal{P}\otimes\mathcal{B}(X)\otimes\mathcal{B}(X)$-measurable function $\mathcal{Q}$ on $\Omega\times[0,T]\times X\times X$ that satisfy the following.
\begin{enumerate}
\renewcommand{\labelenumi}{(\roman{enumi})}
\item
For all $(\omega,s)\in\Omega\times[0,T]$, $\mathcal{Q}_{s,\omega}$ is in $\mathcal{K}^+_{\text{c}}(X)$.
\item
For all $(\omega,s)\in\Omega\times[0,T]$, $\int^t_0\mathcal{Q}_{s,\omega}\,dA_s(\omega)$ is in $\mathcal{K}^+_{\text{c}}(X)$.
\item
For each $x,y\in X$ and $t\in[0,T]$,
\begin{equation*}
	\langle M^x,M^y\rangle_t=\int^t_0\mathcal{Q}_s(x,y)\,dA_s\ \bold{P}\text{-a.s.}
\end{equation*}
\end{enumerate}
\end{assum}
\begin{rem}
In some literatures, the process $A$ is assumed to be only nondecreasing and predictable. In such a case, without loss of generality we may further assume that $A$ is strictly increasing and bounded by replacing $A_t$ by $\arctan(t+A_t)$.
\end{rem}
For $\mathcal{Q}\in\mathcal{K}^+_{\text{c}}(X)$, we can define a corresponding linear mapping (also denoted by $\mathcal{Q}$) from $\mathcal{M}$ into $\mathcal{C}$ by setting
\begin{equation*}
	(\mathcal{Q}\mu)(x)=\int_X\mathcal{Q}(x,y)\,\mu(dy)=\langle\mu,\mathcal{Q}(x,\cdot)\rangle_{\mathcal{M},\mathcal{C}},\ \mu\in\mathcal{M}.
\end{equation*}
Then $\mathcal{Q}$ is symmetric and nonnegative-definite, i.e., $\langle\mu,\mathcal{Q}\nu\rangle_{\mathcal{M},\mathcal{C}}=\langle\nu,\mathcal{Q}\mu\rangle_{\mathcal{M},\mathcal{C}}$ and $\langle\mu,\mathcal{Q}\mu\rangle_{\mathcal{M},\mathcal{C}}\geq0$ for all $\mu,\nu\in\mathcal{M}$. $\mathcal{Q}$ is also (weakly) continuous.\par
Let $\mathcal{D}$ denote the set of all linear combinations of Dirac measures on $X$. For an element $\mu=\sum^n_{i=1}c_i\delta_{x_i}\in\mathcal{D}$, where each $c_i$ is a real constant and $\delta_{x_i}$ is the Dirac measure at $x_i\in X$, we set
\begin{equation*}
	\mathbb{M}(\mu)=\sum^n_{i=1}c_iM^{x_i}.
\end{equation*}
Then, $\mathbb{M}(\mu)$ is a square-integrable martingale. The linear mapping $\mu\mapsto\mathbb{M}(\mu)$ from $\mathcal{D}$ into $\mathcal{H}^2(\bold{P})$ extends uniquely to a continuous linear mapping $\mathbb{M}\colon\mathcal{M}\to\mathcal{H}^2(\bold{P})$. For each $\mu,\nu\in\mathcal{M}$, the cross variation between $\mathbb{M}(\mu)$ and $\mathbb{M}(\nu)$ is given by
\begin{equation*}
	\langle\mathbb{M}(\mu),\mathbb{M}(\nu)\rangle_t=\int^t_0\langle\mu,\mathcal{Q}_s\nu\rangle_{\mathcal{M},\mathcal{C}}\,dA_s.
\end{equation*}
The function $\mathcal{Q}_{s,\omega}$ is called the covariance operator function, while $\int^t_0\mathcal{Q}_{s,\omega}\,dA_s(\omega)$ is called the predictable quadratic variation of the cylindrical martingale $\mathbb{M}$.\par
A simple integrand $H$ is a process of the form
\begin{equation*}
	H=\sum^n_{i=1}h^i\delta_{x_i}\,,
\end{equation*}
where each $h^i$ is a real-valued bounded predictable process and $x_i\in X$. $H$ can be considered as an $\mathcal{M}$-valued process. Define the stochastic integral of $H$ with respect to $\mathbb{M}$ by
\begin{equation}\nonumber
	H\bullet\mathbb{M}=\int^\cdot_0H_s\,d\mathbb{M}_s=\sum^n_{i=1}\int^\cdot_0h^i_s\,dM^{x_i}_s.
\end{equation}
Note that $H\bullet\mathbb{M}\in\mathcal{H}^2(\bold{P})$ and
\begin{align*}
	\bold{E}\left[\left(\int^T_0H_s\,d\mathbb{M}_s\right)^2\right]&=\bold{E}\left[\int^T_0\sum^n_{i,j=1}h^i_sh^j_s\,d\langle M^{x_i},M^{x_j}\rangle_s\right]\\
		&=\bold{E}\left[\int^T_0\sum^n_{i,j=1}h^i_sh^j_s\mathcal{Q}_s(x_i,x_j)\,dA_s\right]\\
		&=\bold{E}\left[\int^T_0\langle H_s,\mathcal{Q}_sH_s\rangle_{\mathcal{M},\mathcal{C}}\,dA_s\right].
\end{align*}
An $\mathcal{M}$-valued process $H$ is called (weakly) predictable if the real-valued processes $\langle H_s,f\rangle_{\mathcal{M},\mathcal{C}}$ are predictable for all $f\in\mathcal{C}$. Note that for an $\mathcal{M}$-valued predictable process $H$, $\langle H_s,\mathcal{Q}_sH_s\rangle_{\mathcal{M},\mathcal{C}}$ is a real-valued predictable process. Let $L^2(\mathbb{M},\mathcal{M})$ denote the set of all $\mathcal{M}$-valued predictable processes $H$ such that
\begin{equation}\label{measure-valued integrand norm}
	\bold{E}\left[\int^T_0\langle H_s,\mathcal{Q}_sH_s\rangle_{\mathcal{M},\mathcal{C}}\,dA_s\right]^{1/2}<\infty.
\end{equation}
\begin{theo}
The set of all simple integrands is dense in $L^2(\mathbb{M},\mathcal{M})$ with respect to the norm (\ref{measure-valued integrand norm}). Thus, the map $H\mapsto H\bullet\mathbb{M}$ can extend continuously from $L^2(\mathbb{M},\mathcal{M})$ to $\mathcal{H}^2(\bold{P})$ uniquely. This extension is linear and
\begin{equation*}
	\langle H\bullet\mathbb{M},K\bullet\mathbb{M}\rangle_t=\int^t_0\langle H_s,\mathcal{Q}_sK_s\rangle_{\mathcal{M},\mathcal{C}}\,dA_s
\end{equation*}
holds for all $H,K\in L^2(\mathbb{M},\mathcal{M})$.
\end{theo}
\begin{rem}
In financial terms, each $H\in L^2(\mathbb{M},\mathcal{M})$ is considered as a measure-valued trading strategy in bond markets; the signed-measure $H_{t,\omega}$ represents the amount of holdings of zero-coupon bonds of all maturities among $X=[0,T^*]$ at time $t$ for an event $\omega$.
\end{rem}
The space $L^2(\mathbb{M},\mathcal{M})$ is not complete as mentioned in \cite{a_DeDonno-Pratelli_04}. We need further extension of integrands as follows.\par
For a symmetric, nonnegative-definite and continuous linear mapping $\mathcal{Q}\colon\mathcal{M}\to\mathcal{C}$, define the scalar product on $\mathcal{Q}(\mathcal{M})$ by
\begin{equation}\label{kernel product}
	(\mathcal{Q}\mu,\mathcal{Q}\nu)_{U_\mathcal{Q}}=\langle\mu,\mathcal{Q}\nu\rangle_{\mathcal{M},\mathcal{C}},\ \mu,\nu\in\mathcal{M}.
\end{equation}
Then $\mathcal{Q}(\mathcal{M})$ admits a unique completion $U_\mathcal{Q}$ in $\mathcal{C}$ with respect to the norm induced by (\ref{kernel product}). This completion $U_\mathcal{Q}$ is a separable Hilbert space and continuously embedded in $\mathcal{C}$. The mapping $\mathcal{Q}\colon\mathcal{M}\to\mathcal{C}$ can extend continuously to the canonical isomorphism from $U'_\mathcal{Q}$ to $U_\mathcal{Q}$, where $U'_\mathcal{Q}$ is the topological dual of $U_\mathcal{Q}$. Moreover, $U'_\mathcal{Q}$ is the completion of $\mathcal{M}/\text{Ker}\mathcal{Q}$ with respect to the norm induced by the scalar product
\begin{equation*}
	(\mu,\nu)_{U'_\mathcal{Q}}=\langle\mu,\mathcal{Q}\nu\rangle_{\mathcal{M},\mathcal{C}},\ \mu,\nu\in\mathcal{M}.
\end{equation*}\par
Accordingly, we can construct a family of Hilbert spaces $U_{t,\omega}=U_{\mathcal{Q}_{t,\omega}}$ and $U'_{t,\omega}=U'_{\mathcal{Q}_{t,\omega}}$ parametrized by $(t,\omega)\in[0,T]\times\Omega$. We call $(U_{t,\omega})_{(t,\omega)\in[0,T]\times\Omega}$ the family of covariance spaces for $\mathbb{M}$.
\begin{defi}
A $U'$-valued process $\mathbb{H}$ is a process on $\Omega\times[0,T]$ such that $\mathbb{H}_t(\omega)\in U'_{t,\omega}$ for all $(t,\omega)\in[0,T]\times\Omega$. We say that such an $\mathbb{H}$ is predictable if the process $(\omega,t)\mapsto(\mathbb{H}_t(\omega),\mu)_{U'_{t,\omega}}$ is predictable for any $\mu\in\mathcal{M}$.
\end{defi}
\begin{rem}\label{remark: CONS}
Let $\{x_1,x_2,\dots\}$ be a countable dense subset of $X$. By a standard orthogonalization procedure, we can construct a sequence of processes $\{e^m\}_{m\in\mathbb{N}}$ of the form $e^m_s=\sum^m_{k=1}\alpha^{m,k}_s\delta_{x_k}$ for each $m\in\mathbb{N}$, where $\alpha^{m,k}$ are real-valued predictable processes, such that $\left\{e^m_{s,\omega}\relmiddle|m\in\mathbb{N}\ \text{with}\ e^m_{s,\omega}\neq0\right\}$ is a complete orthonormal system of $U'_{s,\omega}$ for any $(s,\omega)\in[0,T]\times\Omega$. Then each $U'$-valued predictable process $\mathbb{H}$ can be expanded in $U'_{s,\omega}$ as
\begin{equation*}
	\mathbb{H}_{s,\omega}=\sum_{m\in\mathbb{N}}(\mathbb{H}_{s,\omega},e^m_{s,\omega})_{U'_{s,\omega}}e^m_{s,\omega}.
\end{equation*}
Note that, for each $m\in\mathbb{N}$, the process $(\mathbb{H}_s,e^m_s)_{U'_s}=\sum^m_{k=1}\alpha^{m,k}_s(\mathbb{H}_s,\delta_{x_k})_{U'_s}$ is real-valued and predictable and so is the process $|\mathbb{H}_s|^2_{U'_s}=\sum_{m\in\mathbb{N}}|(\mathbb{H}_s,e^m_s)_{U'_s}|^2$. See \cite{ic_Mikulevicius-Rozovskii_98} for further details.
\end{rem}
Define the set
\begin{equation*}
	L^2(\mathbb{M},U')=\left\{\mathbb{H}\relmiddle|\mathbb{H}\ \text{is a}\ U'\text{-valued predictable process s.t.}\ \bold{E}\left[\int^T_0|\mathbb{H}_s|^2_{U'_s}\,dA_s\right]<\infty\right\}.
\end{equation*}
Then $L^2(\mathbb{M},U')$ is a Hilbert space with the scalar product
\begin{equation*}
	(\mathbb{H},\mathbb{K})_\mathbb{M}=\bold{E}\left[\int^T_0(\mathbb{H}_s,\mathbb{K}_s)_{U'_s}\,dA_s\right],\ \mathbb{H},\mathbb{K}\in L^2(\mathbb{M},U').
\end{equation*}
Let $\|\cdot\|_\mathbb{M}$ denote the corresponding norm.
\begin{theo}[De Donno and Pratelli \cite{a_DeDonno-Pratelli_04}, De Donno \cite{a_DeDonno_07}]
For any $\mathbb{H}\in L^2(\mathbb{M},U')$, there exists a sequence of simple integrands $(H^n)_{n\in\mathbb{N}}$ such that $H^n_{t,\omega}$ converges to $\mathbb{H}_{t,\omega}$ in $U'_{t,\omega}$ for all $(t,\omega)\in[0,T]\times\Omega$ and $(H^n\bullet\mathbb{M})_{n\in\mathbb{N}}$ is a Cauchy sequence in $\mathcal{H}^2(\bold{P})$.
\end{theo}
As a consequence, we can define the stochastic integral $\mathbb{H}\bullet\mathbb{M}=\int^\cdot_0\mathbb{H}_s\,d\mathbb{M}_s$ as the limit of the sequence $(H^n\bullet\mathbb{M})_{n\in\mathbb{N}}$. The mapping $L^2(\mathbb{M},U')\ni\mathbb{H}\mapsto\mathbb{H}\bullet\mathbb{M}\in\mathcal{H}^2(\bold{P})$ is linear. It holds that
\begin{equation*}
	\langle\mathbb{H}\bullet\mathbb{M},\mathbb{K}\bullet\mathbb{M}\rangle_t=\int^t_0(\mathbb{H}_s,\mathbb{K}_s)_{U'_s}\,dA_s
\end{equation*}
and hence
\begin{equation*}
	\|\mathbb{H}\bullet\mathbb{M}\|_{\mathcal{H}^2(\bold{P})}=\|\mathbb{H}\|_\mathbb{M}
\end{equation*}
for any $\mathbb{H},\mathbb{K}\in L^2(\mathbb{M},U')$.
\begin{cor}
The stable subspace generated by $\mathbb{M}=(M^x)_{x\in X}$ in the set of square-integrable martingales coincides with the set of stochastic integrals $\{\mathbb{H}\bullet\mathbb{M}\ |\ \mathbb{H}\in L^2(\mathbb{M},U')\}$. As a consequence, for any $\xi\in L^2(\mathcal{F}_T,\bold{P})$, there exist unique $\mathbb{H}\in L^2(\mathbb{M},U')$ and $N\in\mathcal{H}^2(\bold{P})$ such that $N_0=0$, $N$ is strongly orthogonal to $\mathbb{M}$, and
\begin{equation*}
	\xi=\bold{E}[\xi]+\int^T_0\mathbb{H}_s\,d\mathbb{M}_s+N_T.
\end{equation*}
\end{cor}


\section{BSDEs driven by cylindrical martingales on the space of continuous functions}\label{section: BSDE}
\paragraph{}
In this section, we consider backward stochastic differential equations (BSDEs) driven by square-integrable cylindrical martingales on $\mathcal{C}$. We keep the notations in Section \ref{section: cylindrical mart} and further impose the following assumption.
\begin{assum}\label{assumption: A is conti}
The predictable process $A$ in Assumption \ref{assumption: cylindrical mart} is continuous.
\end{assum}
For a given data $(f,\xi)$ specified below, we consider a BSDE of the form
\begin{equation}\label{cylindricalBSDE}
	Y_t=\xi+\int^T_tf(s,Y_s,\mathbb{H}_s)\,dA_s-\int^T_t\mathbb{H}_s\,d\mathbb{M}_s-\int^T_tdN_s,\ t\in[0,T].
\end{equation}
To refer to this BSDE, we use the notation $\text{BSDE}(f,\xi)$. Here, the data $(f,\xi)$ consists of a driver $f$ and a terminal condition $\xi$ satisfying the following assumptions.
\begin{assum}\label{assumption: BSDEdata}
\begin{enumerate}
\renewcommand{\labelenumi}{(\roman{enumi})}
\item
The driver $f$ is a real-valued function on $\{(\omega,s,y,h)\ |\ \omega\in\Omega,\ s\in[0,T],\ y\in\mathbb{R},\ h\in U'_{s,\omega}\}$ such that;
\begin{itemize}
\item
for any $U'$-valued predictable process $\mathbb{K}$, the function $(\omega,s,y)\mapsto f(\omega,s,y,\mathbb{K}_{s,\omega})$ is $\mathcal{P}\otimes\mathcal{B}(\mathbb{R})$-measurable, and
\item
there exist positive predictable processes $\eta$ and $\theta$ such that
\begin{equation}\label{Lipschitz}
	|f(t,y_1,\mathbb{K}^1_t)-f(t,y_2,\mathbb{K}^2_t)|\leq\eta_t|y_1-y_2|+\theta_t|\mathbb{K}^1_t-\mathbb{K}^2_t|_{U'_t}\ dA_t\otimes d\bold{P}\text{-a.e.}
\end{equation}
for any $y_1,y_2\in\mathbb{R}$ and $U'$-valued processes $\mathbb{K}^1,\mathbb{K}^2$.
\end{itemize}
\item
The terminal condition $\xi$ is a real-valued and $\mathcal{F}_T$-measurable random variable.
\end{enumerate}
\end{assum}
We define $\alpha_t=\sqrt{\eta_t+\theta^2_t}>0$ and $K_t=\int^t_0\alpha^2_s\,dA_s$. Note that both $\alpha$ and $K$ are predictable, and $K$ is continuous and nondecreasing. For each $\beta\geq0$, define the following spaces;
\begin{align*}
	L^2_\beta&=\left\{\zeta\in L^2(\mathcal{F}_T,\bold{P})\relmiddle|\|\zeta\|^2_\beta=\bold{E}\left[e^{\beta K_T}|\zeta|^2\right]<\infty\right\},\\
	L^2_{T,\beta}&=\left\{\phi\relmiddle|
\begin{gathered}
\phi\ \text{is a real-valued and \cadlag\ adapted process s.t.}\\
\|\phi\|^2_{T,\beta}=\bold{E}\left[\int^T_0e^{\beta K_t}|\phi_t|^2\,dA_t\right]<\infty
\end{gathered}
\right\},\\
	L^2_\beta(\mathbb{M},U')&=\left\{\mathbb{K}\in L^2(\mathbb{M},U')\relmiddle|\|\mathbb{K}\|^2_{\mathbb{M},\beta}=\bold{E}\left[\int^T_0e^{\beta K_t}|\mathbb{K}_t|^2_{U'_t}\,dA_t\right]<\infty\right\},\\
	\mathcal{H}^2_\beta(\bold{P})&=\left\{L\in\mathcal{H}^2(\bold{P})\relmiddle|\|L\|^2_{\mathcal{H}^2_\beta(\bold{P})}=\bold{E}\left[\int^T_0e^{\beta K_t}\,d[L]_t\right]<\infty\right\},\\
	\mathcal{S}_\beta&=\left\{(\phi,\mathbb{K},L)\relmiddle|\alpha\phi\in L^2_{T,\beta},\ \mathbb{K}\in L^2_\beta(\mathbb{M},U'),\ \text{and}\ L\in\mathcal{H}^2_\beta(\bold{P})\right\}.
\end{align*}
\begin{rem}\label{remark: BSDE}
\begin{enumerate}
\renewcommand{\labelenumi}{(\roman{enumi})}
\item
If the process $K=\int^\cdot_0\alpha^2_s\,dA_s$ is uniformly bounded, then all the spaces $L^2_\beta$ (resp.\ $L^2_{T,\beta}$, $L^2_\beta(\mathbb{M},U')$, $\mathcal{H}^2_\beta(\bold{P})$) coincide with $L^2$ (resp.\ $L^2_T=L^2_{T,0}$, $L^2(\mathbb{M},U')$, $\mathcal{H}^2(\bold{P})$) and all the norms are equivalent.
\item
For any $\beta\geq0$, we can define the stochastic integral $\mathbb{K}\bullet\mathbb{M}\in\mathcal{H}^2(\bold{P})$ for $\mathbb{K}\in L^2_\beta(\mathbb{M},U')$. Moreover, it can be shown that $\mathbb{K}\bullet\mathbb{M}\in\mathcal{H}^2_\beta(\bold{P})$.
\item
For any $L\in\mathcal{H}^2_\beta(\bold{P})$, since $[L]-\langle L\rangle$ is a martingale and $K$ is predictable, we have that
\begin{equation*}
	\bold{E}\left[\int^T_0e^{\beta K_t}\,d[L]_t\right]=\bold{E}\left[\int^T_0e^{\beta K_t}\,d\langle L\rangle_t\right]<\infty.
\end{equation*}
\item
The set $\mathcal{S}_\beta$ is a Hilbert space with the norm
\begin{equation*}
	\|(\phi,\mathbb{K},L)\|^2_{\mathcal{S}_\beta}=\|\alpha\phi\|^2_{T,\beta}+\|\mathbb{K}\|^2_{\mathbb{M},\beta}+\|L\|^2_{\mathcal{H}^2_\beta(\bold{P})}.
\end{equation*}
\end{enumerate}
\end{rem}
\begin{defi}
Fix any $\beta\geq0$. A data $(f,\xi)$ satisfying Assumption \ref{assumption: BSDEdata} is called $\beta$-standard if $\alpha^{-1}f(\cdot,0,0)\in L^2_{T,\beta}$ and $\xi\in L^2_\beta$.
\end{defi}
\begin{defi}
For a $\beta$-standard data $(f,\xi)$ with $\beta\geq0$, a triple $(Y,\mathbb{H},N)$ in $\mathcal{S}_\beta$ is called a solution of $\text{BSDE}(f,\xi)$ if $N_0=0$, $\langle N,M^x\rangle\equiv0$ for all $x\in X$, and (\ref{cylindricalBSDE}) holds $\bold{P}$-a.s.
\end{defi}
The following theorem is our first main result about the existence and uniqueness of the solution of the infinite-dimensional BSDE (\ref{cylindricalBSDE}).
\begin{theo}\label{cylindricalBSDE existence and uniqueness}
Let $\beta>3$. For a $\beta$-standard data $(f,\xi)$, there exists a unique solution $(Y,\mathbb{H},N)$ of $\text{BSDE}(f,\xi)$ in $\mathcal{S}_\beta$.
\end{theo}
Theorem \ref{cylindricalBSDE existence and uniqueness} is proved by a slight modification of the arguments in Carbone, Ferrario, and Santacrose \cite{a_Carbone-_07} and El Karoui and Huang \cite{ic_ElKaroui-Huang_97}. For completeness, we give the proof of Theorem \ref{cylindricalBSDE existence and uniqueness} by applying \cite{a_Carbone-_07} in our setting.\par
Firstly, consider the case when a driver $f$ depends neither on $Y$ nor on $\mathbb{K}$.
\begin{lemm}\label{BSDE lemma1}
Fix any $\beta>0$ and let $\xi\in L^2_\beta$ and $f(\omega,t,\cdot,\cdot)\equiv g(\omega,t)$ with $\alpha^{-1}g\in L^2_{T,\beta}$. Then there exists a unique solution of $\text{BSDE}(f,\xi)$ in $\mathcal{S}_\beta$.
\end{lemm}
\begin{proof}
Uniqueness: It suffices to show that the triple $(0,0,0)$ is the unique solution of $\text{BSDE}(0,0)$. Let $(Y,\mathbb{H},N)\in\mathcal{S}_\beta$ be a solution of $\text{BSDE}(0,0)$. Then $Y_t=-\int^T_t\mathbb{H}_s\,d\mathbb{M}_s-\int^T_tdN_t$ $\bold{P}$-a.s.\ for all $t\in[0,T]$. Since the stochastic integrals in the right hand side are backward martingales and the process $Y$ is adapted, taking the conditional expectation given $\mathcal{F}_t$ for each $t\in[0,T]$ yields that $Y\equiv0$. Then we have that $N_T=-\int^T_0\mathbb{H}_s\,d\mathbb{M}_s$ $\bold{P}$-a.s. Since $N$ is strongly orthogonal to $\mathbb{M}$, we see that $\mathbb{H}=0$ in $L^2(\mathbb{M},U')$ and $N\equiv0$.\\
Existence: Firstly, for any process $h$, constant $p>0$, and $t\in[0,T]$, it holds that
\begin{align}\label{sublem}\nonumber
	\left(\int^T_th_s\,dA_s\right)^2&=\left(\int^T_te^{-pK_s/2}\alpha_se^{pK_s/2}\alpha^{-1}_sh_s\,dA_s\right)^2\\\nonumber
		&\leq\left(\int^T_te^{-pK_s}\alpha^2_s\,dA_s\right)\left(\int^T_te^{pK_s}\alpha^{-2}_sh^2_s\,dA_s\right)\\
		&=\frac{1}{p}\left(e^{-pK_t}-e^{-pK_T}\right)\int^T_te^{pK_s}\alpha^{-2}_sh^2_s\,dA_s.
\end{align}
In particular, since $\alpha^{-1}g\in L^2_{T,\beta}$ by the assumption, setting $ h=g$, $p=\beta$, and $t=0$ in the above estimate yields that $\int^T_0g_s\,dA_s\in L^2(\mathcal{F}_T,\bold{P})$. Set
\begin{equation*}
	\widetilde{L}_t=\bold{E}\left[\xi+\int^T_0g_s\,dAs\relmiddle|\mathcal{F}_t\right],\ t\in[0,T]
\end{equation*}
and denote the right-continuous version of $\widetilde{L}$ by $L$. Then $L$ is in $\mathcal{H}^2(\bold{P})$. Hence, there exist $\mathbb{H}\in L^2(\mathbb{M},U')$ and a martingale $N\in\mathcal{H}^2(\bold{P})$ such that $N_0=0$, $N$ is strongly orthogonal to $\mathbb{M}$, and
\begin{equation*}
L=L_0+\int^\cdot_0\mathbb{H}_s\,d\mathbb{M}_s+N.
\end{equation*}
Set $Y_t=L_t-\int^t_0g_s\,dA_s$. Then $Y$ is right-continuous and adapted. Furthermore, $Y$ can be written as
\begin{equation*}
	Y_t=\xi+\int^T_tg_s\,dA_s-\int^T_t\mathbb{H}_s\,d\mathbb{M}_s-\int^T_tdN_s,\ t\in[0,T].
\end{equation*}
We show that the triple $(Y,\mathbb{H},N)$ is in $\mathcal{S}_\beta$. Note that
\begin{align*}
	Y^2_t&=\bold{E}\left[\xi+\int^T_tg_s\,dA_s\relmiddle|\mathcal{F}_t\right]^2\leq\bold{E}\left[\left(\xi+\int^T_tg_s\,dA_s\right)^2\relmiddle|\mathcal{F}_t\right]\\
		&\leq2\bold{E}\left[\xi^2\relmiddle|\mathcal{F}_t\right]+2\bold{E}\left[\left(\int^T_tg_s\,dA_s\right)^2\relmiddle|\mathcal{F}_t\right]\\
		&\leq2\bold{E}\left[\xi^2\relmiddle|\mathcal{F}_t\right]+\frac{4}{\beta}e^{-\beta K_t/2}\bold{E}\left[\int^T_te^{\beta K_s/2}\alpha^{-2}_sg^2_s\,dA_s\relmiddle|\mathcal{F}_t\right],\ t\in[0,T],
\end{align*}
where we used the Jensen's inequality in the first inequality and the estimate (\ref{sublem}) with $p=\beta/2$ in the third inequality. We see that
\begin{align*}
	\|\alpha Y\|^2_{T,\beta}&=\bold{E}\left[\int^T_0e^{\beta K_t}\alpha^2_tY^2_t\,dA_t\right]\\
		&\leq2\bold{E}\left[\int^T_0\left(e^{\beta K_t}\alpha^2_t\xi^2+\frac{2}{\beta}e^{\beta K_t/2}\alpha^2_t\int^T_te^{\beta K_s/2}\alpha^{-2}_sg^2_s\,dA_s\right)\,dA_t\right]\\
		&=2\bold{E}\left[\xi^2\int^T_0e^{\beta K_t}\,dK_t\right]+\frac{4}{\beta}\bold{E}\left[\int^T_0e^{\beta K_s/2}\alpha^{-2}_sg^2_s\left(\int^s_0e^{\beta K_t/2}\,dK_t\right)\,dA_s\right]\\
		&\leq\frac{2}{\beta}\|\xi\|^2_\beta+\frac{8}{\beta^2}\|\alpha^{-1}g\|^2_{T,\beta}<\infty,
\end{align*}
and hence $\alpha Y\in L^2_{T,\beta}$. Next, let us prove that $\mathbb{H}\in L^2_\beta(\mathbb{M},U')$ and $N\in\mathcal{H}^2_\beta(\bold{P})$. Set $n=\int^\cdot_0\mathbb{H}_s\,d\mathbb{M}_s+N\in\mathcal{H}^2(\bold{P})$. Since $N$ is strongly orthogonal to $\mathbb{M}$, it holds that $\langle n\rangle=\int^\cdot_0|\mathbb{H}_s|^2_{U'_s}\,dA_s+\langle N\rangle$. Noting that the increasing process $K$ is continuous by Assumption \ref{assumption: A is conti}, It\^{o}'s formula yields that
\begin{equation*}
	d(e^{\beta K_t}(\langle n\rangle_T-\langle n\rangle_t))=\beta e^{\beta K_t}(\langle n\rangle_T-\langle n\rangle_t)dK_t-e^{\beta K_t}d\langle n\rangle_t,
\end{equation*}
and hence
\begin{align}\label{estimate1}\nonumber
	\|\mathbb{H}\|^2_{\mathbb{M},\beta}+\|N\|^2_{\mathcal{H}^2_\beta(\bold{P})}&=\bold{E}\left[\int^T_0e^{\beta K_t}\,d\langle n\rangle_t\right]\\
	&=\bold{E}[\langle n\rangle_T]+\beta\bold{E}\left[\int^T_0e^{\beta K_t}(\langle n\rangle_T-\langle n\rangle_t)\,dK_t\right]\\
	&=\|\mathbb{H}\|^2_\mathbb{M}+\|N\|^2_{\mathcal{H}^2(\bold{P})}+\beta\bold{E}\left[\int^T_0e^{\beta K_t}\bold{E}\left[\langle n\rangle_T-\langle n\rangle_t\middle|\mathcal{F}_t\right]\,dK_t\right].
\end{align}
Furthermore, we have
\begin{align}\label{estimate2}\nonumber
	\bold{E}\left[\langle n\rangle_T-\langle n\rangle_t\relmiddle|\mathcal{F}_t\right]&=\bold{E}\left[(n_T-n_t)^2\relmiddle|\mathcal{F}_t\right]\\\nonumber
		&=\bold{E}\left[\left(\xi-Y_t+\int^T_tg_s\,dA_s\right)^2\relmiddle|\mathcal{F}_t\right]\\\nonumber
		&=\bold{E}\left[\left(\xi+\int^T_tg_s\,dA_s\right)^2\relmiddle|\mathcal{F}_t\right]-\bold{E}\left[\xi+\int^T_tg_s\,dA_s\relmiddle|\mathcal{F}_t\right]^2\\
		&\leq2\bold{E}\left[\xi^2+\left(\int^T_tg_s\,dA_s\right)^2\relmiddle|\mathcal{F}_t\right],
\end{align}
where we used the identity $Y_t=\bold{E}[\xi+\int^T_tg_s\,dA_s\ |\ \mathcal{F}_t]$. Hence, in the same way as above estimates, we can show that the third term in (\ref{estimate1}) is finite. Hence $\|H\|^2_{\mathbb{M},\beta}+\|N\|^2_{\mathcal{H}^2_\beta(\bold{P})}<\infty$ and all assertions are proved.
\end{proof}
\begin{lemm}\label{lemma: mart}
Fix a $\beta$-standard data with $\beta>0$ and let $(Y,\mathbb{H},N)\in\mathcal{S}_\beta$ be a solution of $\text{BSDE}(f,\xi)$. Then, $\sup_{t\in[0,T]}|e^{\frac{\beta}{2}K_t}Y_t|$ is in $L^2(\bold{P})$. Consequently, for any $L\in\mathcal{H}^2_\beta(\bold{P})$, the stochastic integral $\int^\cdot_0e^{\beta K_s}Y_{s-}\,dL_s$ is a martingale.
\end{lemm}
\begin{proof}
By using the estimate (\ref{sublem}) with $h_s=f(s,Y_s,\mathbb{H}_s)$ and $p=\beta$, we can show that
\begin{align*}
	e^{\beta K_t}Y^2_t&=\left(\bold{E}\left[e^{\frac{\beta}{2}K_t}\xi\relmiddle|\mathcal{F}_t\right]+\bold{E}\left[e^{\frac{\beta}{2}K_t}\int^T_tf(s,Y_s,\mathbb{H}_s)\,dA_s\relmiddle|\mathcal{F}_t\right]\right)^2\\
		&\leq2\bold{E}\left[e^{\frac{\beta}{2}K_t}\xi\relmiddle|\mathcal{F}_t\right]^2+2\bold{E}\left[e^{\frac{\beta}{2}K_t}\int^T_tf(s,Y_s,\mathbb{H}_s)\,dA_s\relmiddle|\mathcal{F}_t\right]^2\\
		&\leq2\bold{E}\left[e^{\frac{\beta}{2}K_T}\xi\relmiddle|\mathcal{F}_t\right]^2+\frac{2}{\beta}\bold{E}\left[\left(\int^T_0e^{\beta K_s}\alpha^{-2}_sf^2(s,Y_s,\mathbb{H}_s)\,dA_s\right)^{1/2}\relmiddle|\mathcal{F}_t\right]^2.
\end{align*}
Hence,
\begin{align*}
	\bold{E}\left[\sup_{t\in[0,T]}e^{\beta K_t}Y^2_t\right]&\leq2\bold{E}\left[\sup_{t\in[0,T]}\bold{E}\left[e^{\beta K_T/2}\xi\relmiddle|\mathcal{F}_t\right]^2\right]\\
		&\hspace{6mm}+\frac{2}{\beta}\bold{E}\left[\sup_{t\in[0,T]}\bold{E}\left[\left(\int^T_0e^{\beta K_s}\alpha^{-2}_sf^2(s,Y_s,\mathbb{H}_s)\,dA_s\right)^{1/2}\relmiddle|\mathcal{F}_t\right]^2\right]\\
		&\leq8\bold{E}\left[e^{\beta K_T}\xi^2\right]+\frac{8}{\beta}\bold{E}\left[\int^T_0e^{\beta K_s}\alpha^{-2}_sf^2(s,Y_s,\mathbb{H}_s)\,dA_s\right]\\
		&\leq8\|\xi\|^2_\beta+\frac{24}{\beta}\left(\|\alpha^{-1}f(\cdot,0,0)\|^2_{T,\beta}+\|\alpha Y\|^2_{T,\beta}+\|\mathbb{H}\|^2_{\mathbb{M},\beta}\right)\\
		&<\infty,
\end{align*}
where we used Doob's inequality in the second inequality and the Lipschitz condition (\ref{Lipschitz}) in the third inequality.\par
Let $L\in\mathcal{H}^2_\beta(\bold{P})$ be given. Noting Remark \ref{remark: BSDE} (\rnum{3}) and that $K$ is continuous, we have that
\begin{align*}
	\bold{E}\left[\left(\int^T_0e^{2\beta K_t}Y^2_{t-}\,d[L]_t\right)^{1/2}\right]&\leq\bold{E}\left[\left(\int^T_0e^{\beta K_t}\,d[L]_t\right)^{1/2}\sup_{t\in[0,T]}e^{\beta K_t/2}|Y_t|\right]\\
		&\leq\|L\|_{\mathcal{H}^2_\beta(\bold{P})}\cdot\left\|\sup_{t\in[0,T]}e^{\beta K_t/2}|Y_t|\right\|_{L^2(\bold{P})}<\infty,
\end{align*}
and hence the stochastic integral $\int^\cdot_0e^{\beta K_s}Y_{s-}\,dL_s$ is a martingale.
\end{proof}
\begin{proof}[Proof of Theorem \ref{cylindricalBSDE existence and uniqueness}]
Define a map $\Phi\colon\mathcal{S}_\beta\to\mathcal{S}_\beta$ by $\Phi(Y,\mathbb{H},N)=(\widehat{Y},\widehat{\mathbb{H}},\widehat{N})$, where $(\widehat{Y},\widehat{\mathbb{H}},\widehat{N})$ is the unique solution of the BSDE
\begin{equation*}
	\widehat{Y}_t=\xi+\int^T_tf(s,Y_s,\mathbb{H}_s)\,dA_s-\int^T_t\widehat{\mathbb{H}}_s\,d\mathbb{M}_s-\int^T_td\widehat{N}_s,\ t\in[0,T].
\end{equation*}
From Lemma \ref{BSDE lemma1}, since the process $(\alpha^{-1}_sf(s,Y_s,\mathbb{H}_s))_{s\in[0,T]}$ is in $L^2_{T,\beta}$ for any $(Y,\mathbb{H},N)\in\mathcal{S}_\beta$, the map $\Phi$ is well-defined. To prove Theorem \ref{cylindricalBSDE existence and uniqueness}, it suffices to show that $\Phi$ is a contraction mapping on $\mathcal{S}_\beta$.\par
Fix two elements $(Y^1,\mathbb{H}^1,N^1),\ (Y^2,\mathbb{H}^2,N^2)\in\mathcal{S}_\beta$ and consider $(\widehat{Y}^i,\widehat{\mathbb{H}}^i,\widehat{N}^i)=\Phi(Y^i,\mathbb{H}^i,N^i)$ for each $i=1,2$. Define $\delta Y=Y^1-Y^2$ and $\delta\widehat{Y}=\widehat{Y}^1-\widehat{Y}^2$. Define also $\delta\mathbb{H}$, $\delta\widehat{\mathbb{H}}$, $\delta N$, and $\delta\widehat{N}$ in the same manner. Then $\delta\widehat{Y}$ satisfies
\begin{equation*}
	\delta\widehat{Y}_t=\int^T_t(f(s,Y^1_s,\mathbb{H}^1_s)-f(s,Y^2_s,\mathbb{H}^2_s))\,dA_s-\int^T_t\delta\widehat{\mathbb{H}}_s\,d\mathbb{M}_s-\int^T_td\delta\widehat{N}_s,\ t\in[0,T]
\end{equation*}
for $\bold{P}$-a.s. By noting that $K=\int^\cdot_0\alpha^2_s\,dA_s$ as well as $A$ is continuous by Assumption \ref{assumption: A is conti}, It\^{o}'s formula implies that
\begin{align*}
	d|\delta\widehat{Y}_t|^2&=-2\delta\widehat{Y}_t(f(t,Y^1_t,\mathbb{H}^1_t)-f(t,Y^2_t,\mathbb{H}^2_t))dA_t+d\left[\int^\cdot_0\delta\widehat{\mathbb{H}}_s\,d\mathbb{M}_s\right]_t+d[\delta\widehat{N}]_t\\
		&\hspace{6mm}+2\delta\widehat{Y}_{t-}d\left(\int^t_0\delta\widehat{\mathbb{H}}_s\,d\mathbb{M}_s+\delta\widehat{N}_t\right)
\end{align*}
and that
\begin{equation*}
	d(e^{\beta K_t}|\delta\widehat{Y}_t|^2)=e^{\beta K_t}(\beta\alpha^2_t|\delta\widehat{Y}_t|^2dA_t+d|\delta\widehat{Y}_t|^2).
\end{equation*}
Since $\delta\widehat{\mathbb{H}}\in L^2_\beta(\mathbb{M},U')$ and $\delta\widehat{N}\in\mathcal{H}^2_\beta(\bold{P})$, Lemma \ref{lemma: mart} implies that the stochastic integral $\int^\cdot_0e^{\beta K_t}\delta\widehat{Y}_{t-}\,d(\int^t_0\delta\widehat{\mathbb{H}}_s\,d\mathbb{M}_s+\delta\widehat{N}_t)$ is a martingale. Hence, integrating  on $[0,T]$ and taking expectations yield that
\begin{align*}
	&\bold{E}\left[e^{\beta K_T}|\delta\widehat{Y}_T|^2\right]-\bold{E}\left[|\delta\widehat{Y}_0|^2\right]\\
		&=\bold{E}\left[\int^T_0e^{\beta K_t}(\beta\alpha^2_t|\delta\widehat{Y}_t|^2-2\delta\widehat{Y_t}(f(t,Y^1_t,\mathbb{H}^1_t)-f(t,Y^2_t,\mathbb{H}^2_t)))\,dA_t\right]\\
		&\hspace{6mm}+\bold{E}\left[\int^T_0e^{\beta K_t}\left(d\left\langle\int^\cdot_0\delta\widehat{\mathbb{H}}_s\,d\mathbb{M}_s\right\rangle_t+d\langle\delta\widehat{N}\rangle_t\right)\right]\\
		&=\bold{E}\left[\int^T_0e^{\beta K_t}(\beta\alpha^2_t|\delta\widehat{Y}_t|^2-2\delta\widehat{Y}_t(f(t,Y^1_t,\mathbb{H}^1_t)-f(t,Y^2_t,\mathbb{H}^2_t)))\,dA_t\right]\\
		&\hspace{6mm}+\bold{E}\left[\int^T_0e^{\beta K_t}\left(|\delta\widehat{\mathbb{H}}_t|^2_{U'_t}\,dA_t+d\langle\delta\widehat{N}\rangle_t\right)\right],
\end{align*}
where we used the relation $\langle\int^\cdot_0\delta_n\mathbb{H}_s\,d\mathbb{M}_s\rangle_t=\int^t_0|\delta_n\mathbb{H}_s|^2_{U'_s}\,dA_s$, $t\in[0,T]$, in the second equality. Since $\delta\widehat{Y}_T=0$, it holds that
\begin{align*}
	&\beta\|\alpha\delta\widehat{Y}\|^2_{T,\beta}+\|\delta\widehat{\mathbb{H}}\|^2_{\mathbb{M},\beta}+\|\delta\widehat{N}\|^2_{\mathcal{H}^2_\beta(\bold{P})}\\
	&\leq\bold{E}\left[\int^T_0e^{\beta K_t}2|\delta\widehat{Y}_t|\cdot|f(t,Y^1_t,\mathbb{H}^1_t)-f(t,Y^2_t,\mathbb{H}^2_t)|\,dA_t\right]\\
	&\leq\bold{E}\left[\int^T_0e^{\beta K_t}2|\delta\widehat{Y}_t|\cdot(\eta_t|\delta{Y}_t|+\theta_t|\delta\mathbb{H}_t|_{U'_t})\,dA_t\right]\\
	&\leq\bold{E}\left[\int^T_0e^{\beta K_t}\left(2\mu^2\alpha^2_t|\delta\widehat{Y}_t|^2+\frac{\alpha^2_t}{\mu^2}|\delta Y_t|^2+\frac{1}{\mu^2}|\delta{\mathbb{H}}_t|^2_{U'_t}\right)\,dA_t\right]
\end{align*}
for any constant $\mu>0$, where we used the Lipschitz condition (\ref{Lipschitz}) and the trivial inequality
\begin{equation}\label{trivial inequality}
	2ab\leq\mu^2a^2+\frac{b^2}{\mu^2}\ \text{for all}\ a,b\in\mathbb{R}\ \text{and}\ \mu>0.
\end{equation}
Hence, we have
\begin{equation*}
	(\beta-2\mu^2)\|\alpha\delta\widehat{Y}\|^2_{T,\beta}+\|\delta\widehat{\mathbb{H}}\|^2_{\mathbb{M},\beta}+\|\delta\widehat{N}\|^2_{\mathcal{H}^2_\beta(\bold{P})}\leq\frac{1}{\mu^2}(\|\alpha\delta Y\|^2_{T,\beta}+\|\delta\mathbb{H}\|^2_{\mathbb{M},\beta})
\end{equation*}
for any constant $\mu>0$. Now, since $\beta>3$, we can choose a constant $\mu>0$ so that $\beta-2\mu^2=1$, and then $\frac{1}{\mu^2}=\frac{2}{\beta-1}\in(0,1)$. Taking such $\mu>0$ yields that $\|(\delta\widehat{Y},\delta\widehat{\mathbb{H}},\delta\widehat{N})\|^2_{\mathcal{S}_\beta}\leq\frac{2}{\beta-1}\|(\delta Y,\delta\mathbb{H},\delta N)\|^2_{\mathcal{S}_\beta}$, and hence the map $\Phi$ is a contraction mapping on $\mathcal{S}_\beta$.
\end{proof}
Next, we show an approximation result of the unique solution $(Y,\mathbb{H},N)\in\mathcal{S}_\beta$ of $\text{BSDE}(f,\xi)$ by a sequence of solutions of corresponding finite-dimensional BSDEs.\par
Let $\{x_1,x_2,\dots\}$ be a countable dense subset of $X$. For each $n\in\mathbb{N}$, let $\mathbb{M}^n$ be the $n$-dimensional square-integrable martingale $\mathbb{M}^n=(M^{x_1},\dots,M^{x_n})^{\text{tr}}$. Here, $(\cdot)^{\text{tr}}$ denotes transposition of a (finite-dimensional) vector, so that $\mathbb{M}^n$ is a column vector. Fix any $n\in\mathbb{N}$. The space of $n$-dimensional integrands is defined by 
\begin{equation*}
	L^2(\mathbb{M}^n,\mathbb{R}^n)=\left\{H\relmiddle|H\ \text{is}\ \mathbb{R}^n\text{-valued predictable and}\ \|H\|^2_{\mathbb{M}^n}=\bold{E}\left[\int^T_0H^{\text{tr}}_td\langle\mathbb{M}^n\rangle_tH_t\right]<\infty\right\}.
\end{equation*}
Then $L^2(\mathbb{M}^n,\mathbb{R}^n)$ is a Hilbert space. Furthermore, for any $H\in L^2(\mathbb{M}^n,\mathbb{R}^n)$, we have that $H\bullet\mathbb{M}^n\in\mathcal{H}^2(\bold{P})$ and $\|H\bullet\mathbb{M}^n\|_{\mathcal{H}^2(\bold{P})}=\|H\|_{\mathbb{M}^n}$. The stable subspace generated by $\mathbb{M}^n$ in the set of square-integrable martingales coincides with the set of $n$-dimensional vector stochastic integrals $\{H\bullet\mathbb{M}\ |\ H\in L^2(\mathbb{M}^n,\mathbb{R}^n)\}$. As a consequence, for any random variable $\xi\in L^2(\mathcal{F}_T,\bold{P})$, there exist a unique $H\in L^2(\mathbb{M}^n,\mathbb{R}^n)$ and $N\in\mathcal{H}^2(\bold{P})$ such that $N_0=0$, $N$ is strongly orthogonal to $\mathbb{M}^n$, and
\begin{equation}\label{finite-dim decomposition}
	\xi=\bold{E}[\xi]+\int^T_0H_s\,d\mathbb{M}^n_s+N_T.
\end{equation}
Note that the predictable quadratic variation of $\mathbb{M}^n$ can be written as
\begin{equation*}
	\langle\mathbb{M}^n\rangle_t=\int^t_0\Sigma^n_s\,dA_s,\ t\in[0,T],
\end{equation*}
where $\Sigma^n_{s,\omega}=(\Sigma^{n,(i,j)}_{s,\omega})_{i,j=1,\dots,n}$ with $\Sigma^{n,(i,j)}_{s,\omega}=\mathcal{Q}_{s,\omega}(x_i,x_j)$. We can identify each element $H=(H^1,\dots,H^n)^{\text{tr}}\in L^2(\mathbb{M}^n,\mathbb{R}^n)$ with an $\mathcal{M}$-valued predictable process (also denoted by $H$) by setting $H=\sum^n_{i=1}H^i\delta_{x_i}$. Then we have that
\begin{align*}
	\|H\|^2_{\mathbb{M}^n}&=\bold{E}\left[\int^T_0H^{\text{tr}}_td\langle\mathbb{M}^n\rangle_tH_t\right]\\
		&=\bold{E}\left[\int^T_0H^{\text{tr}}_t\Sigma^n_tH_t\,dA_t\right]\\
		&=\bold{E}\left[\int^T_0\sum^n_{i,j=1}H^i_t\mathcal{Q}_t(x_i,x_j)H^j_t\,dA_t\right]\\
		&=\bold{E}\left[\int^T_0\langle H_t,\mathcal{Q}_tH_t\rangle_{\mathcal{M},\mathcal{C}}\,dA_t\right]\\
		&=\bold{E}\left[\int^T_0|H_t|^2_{U'_t}\,dA_t\right]\\
		&=\|H\|^2_\mathbb{M}<\infty.
\end{align*}
Hence, we see that the measure-valued process $H$ is in $L^2(\mathbb{M},U')$ and the space $L^2(\mathbb{M}^n,\mathbb{R}^n)$ is isometric to a subspace of $L^2(\mathbb{M},U')$.\par
For each $n\in\mathbb{N}$ and a $\beta$-standard data $(f,\xi)$, consider the finite-dimensional BSDE
\begin{equation}\label{finite-dimBSDE}
	Y^n_t=\xi+\int^T_tf(s,Y^n_s,H^n_s)\,dA_s-\int^T_tH^n_s\,d\mathbb{M}^n_s-\int^T_tdN^n_s,\ t\in[0,T].
\end{equation}
We refer this BSDE to $\text{BSDE}^n(f,\xi)$. Let $L^2_\beta(\mathbb{M}^n,\mathbb{R}^n)$ and $\mathcal{S}^n_\beta$ be corresponding finite-dimensional solution spaces, i.e.,
\begin{equation*}
	L^2_{\beta}(\mathbb{M}^n,\mathbb{R}^n)=\left\{H\in L^2(\mathbb{M}^n,\mathbb{R}^n)\relmiddle|\|H\|^2_{\mathbb{M}^n,\beta}=\bold{E}\left[\int^T_0e^{\beta K_t}H^{\text{tr}}_t\,d\langle\mathbb{M}^n\rangle_tH_t\right]<\infty\right\}
\end{equation*}
and
\begin{equation*}
	\mathcal{S}^n_\beta=\left\{(\phi,K,L)\relmiddle|\alpha\phi\in L^2_{T,\beta},\ K\in L^2_\beta(\mathbb{M}^n,\mathbb{R}^n),\ \text{and}\ L\in\mathcal{H}^2_\beta(\bold{P})\right\}.
\end{equation*}
We can see that $L^2_\beta(\mathbb{M}^n,\mathbb{R}^n)\subset L^2(\mathbb{M},U')$ and $\mathcal{S}^n_\beta\subset\mathcal{S}_\beta$ by identifying each finite-dimensional integrand $H$ with the corresponding measure-valued integrand, for any $n\in\mathbb{N}$ and $\beta\geq0$.
\begin{defi}
For $n\in\mathbb{N}$ and a $\beta$-standard data $(f,\xi)$ with $\beta\geq0$, a triple $(Y^n,H^n,N^n)\in\mathcal{S}^n_\beta$ is called a solution of $\text{BSDE}^n(f,\xi)$ if $N^n_0=0$, $\langle N^n,M^{x_i}\rangle\equiv0$ for $i=1,\dots,n$ and (\ref{finite-dimBSDE}) holds $\bold{P}$-a.s.
\end{defi}
Replacing $\mathbb{M}$ by $\mathbb{M}^n$ and noting the representation formula (\ref{finite-dim decomposition}), we can show the following proposition in the same way as Theorem \ref{cylindricalBSDE existence and uniqueness}.
\begin{theo}\label{finite-dimBSDE existence and uniqueness}
Let $\beta>3$. For $n\in\mathbb{N}$ and a $\beta$-standard data $(f,\xi)$, there exists a unique solution $(Y^n,H^n,N^n)$ of $\text{BSDE}^n(f,\xi)$ in $\mathcal{S}^n_\beta$.
\end{theo}
The following theorem is our second main result about an approximation method of the unique solution of the infinite-dimensional $\text{BSDE}(f,\xi)$.
\begin{theo}\label{BSDEapproximation}
Let $(f,\xi)$ be a $\beta$-standard data with $\beta>3$. Let $(Y,\mathbb{H},N)\in\mathcal{S}_\beta$ be the unique solution of $\text{BSDE}(f,\xi)$ and $(Y^n,H^n,N^n)\in\mathcal{S}^n_\beta$ be the unique solution of $\text{BSDE}^n(f,\xi)$ for each $n\in\mathbb{N}$. Then we have that
\begin{equation*}
	\|(Y,\mathbb{H},N)-(Y^n,H^n,N^n)\|_{\mathcal{S}_\beta}\to0\ \text{as}\ n\to\infty.
\end{equation*}
Moreover, there exists a constant $\gamma>0$ depending only on $\beta$ such that
\begin{equation}\label{gamma1}
	\|\alpha(Y-Y^n)\|^2_{T,\beta}+\|\mathbb{H}-H^n\|^2_{\mathbb{M},\beta}\leq\gamma\|N-N^n\|^2_{\mathcal{H}^2_\beta(\bold{P})}
\end{equation}
for all $n\in\mathbb{N}$.
\end{theo}
\begin{proof}
For each $n\in\mathbb{N}$, define $\delta_nY=Y-Y^n$, $\delta_n\mathbb{H}=\mathbb{H}-H^n$, and $\delta_nN=N-N^n$. Then $\delta_n\mathbb{H}\in L^2_\beta(\mathbb{M},U')$, $\delta_nN\in\mathcal{H}^2_\beta(\bold{P})$, and $\delta_nY$ satisfies
\begin{equation*}
	\delta_nY_t=\int^T_t(f(s,Y_s,\mathbb{H}_s)-f(s,Y^n_s,H^n_s))\,dA_s-\int^T_t\delta_n\mathbb{H}_s\,d\mathbb{M}_s-\int^T_td\delta_nN_s
\end{equation*}
for all $t\in[0,T]$, $\bold{P}$-a.s. By noting that $K=\int^\cdot_0\alpha^2_s\,dA_s$ as well as $A$ is continuous by Assumption \ref{assumption: A is conti}, It\^{o}'s formula implies that
\begin{align*}
	d|\delta_nY_t|^2&=-2\delta_nY_t(f(t,Y_t,\mathbb{H}_t)-f(t,Y^n_t,H^n_t))\,dA_t\\
		&\hspace{6mm}+d\left[\int^\cdot_0\delta_n\mathbb{H}_s\,d\mathbb{M}_s\right]_t+d[\delta_nN]_t+2d\left[\int^\cdot_0\delta_n\mathbb{H}_s\,d\mathbb{M}_s,\delta_nN\right]_t\\
		&\hspace{6mm}+2\delta_nY_{t-}\delta_n\mathbb{H}_s\,d\mathbb{M}_t+2\delta_nY_{t-}\,d\delta_nN_t
\end{align*}
and that
\begin{equation}\label{in the proof of the main theorem}
	d(e^{\beta K_t}|\delta_nY_t|^2)=e^{\beta K_t}(\beta\alpha^2_t|\delta_nY_t|^2dA_t+d|\delta_nY_t|^2).
\end{equation}
Note that the processes $2\int^\cdot_0e^{\beta K_t}\delta_nY_{t-}\delta_n\mathbb{H}_t\,d\mathbb{M}_t$ and $2\int^\cdot_0e^{\beta K_t}\delta_nY_{t-}\,d\delta_nN_t$ are martingales. Hence, by integrating (\ref{in the proof of the main theorem}) on $[0,T]$ and taking expectations, we have
\begin{align*}
	&\bold{E}\left[e^{\beta K_T}|\delta_nY_T|^2\right]-\bold{E}\left[|\delta_nY_0|^2\right]\\
		&=\bold{E}\left[\int^T_0e^{\beta K_t}(\beta\alpha^2_t|\delta_nY_t|^2-2\delta_nY_t(f(t,Y_t,\mathbb{H}_t)-f(t,Y^n_t,H^n_t)))\,dA_t\right]\\
		&\hspace{6mm}+\bold{E}\left[\int^T_0e^{\beta K_t}\left(d\left\langle\int^\cdot_0\delta_n\mathbb{H}_s\,d\mathbb{M}_s\right\rangle_t+d\langle\delta_nN\rangle_t+2d\left\langle\int^\cdot_0\delta_n\mathbb{H}_s\,d\mathbb{M}_s,\delta_nN\right\rangle_t\right)\right]\\
		&=\bold{E}\left[\int^T_0e^{\beta K_t}(\beta\alpha^2_t|\delta_nY_t|^2-2\delta_nY_t(f(t,Y_t,\mathbb{H}_t)-f(t,Y^n_t,H^n_t)))\,dA_t\right]\\
		&\hspace{6mm}+\bold{E}\left[\int^T_0e^{\beta K_t}\left(|\delta_n\mathbb{H}_t|^2_{U'_t}\,dA_t+d\langle\delta_nN\rangle_t+2d\left\langle\int^\cdot_0\delta_n\mathbb{H}_s\,d\mathbb{M}_s,\delta_nN\right\rangle_t\right)\right],
\end{align*}
where we used the relation $\langle\int^\cdot_0\delta_n\mathbb{H}_s\,d\mathbb{M}_s\rangle_t=\int^t_0|\delta_n\mathbb{H}_s|^2_{U'_s}\,dA_s$, $t\in[0,T]$, in the second equality. Since $\delta_nY_T=0$, we then have
\begin{align}\label{BSDEinequality}\nonumber
	&\beta\|\alpha\delta_nY\|^2_{T,\beta}+\|\delta_n\mathbb{H}\|^2_{\mathbb{M},\beta}+\|\delta_nN\|^2_{\mathcal{H}^2_\beta(\bold{P})}\\\nonumber
	&\leq\bold{E}\left[\int^T_0e^{\beta K_t}2|\delta_nY_t|\cdot|f(t,Y_t,\mathbb{H}_t)-f(t,Y^n_t,H^n_t)|\,dA_t\right]\\
	&\hspace{6mm}+2\bold{E}\left[\int^T_0e^{\beta K_t}\left|d\left\langle\int^\cdot_0\delta_n\mathbb{H}_s\,d\mathbb{M}_s,\delta_nN\right\rangle_t\right|\right].
\end{align}
Since the driver $f$ satisfies the Lipschitz condition (\ref{Lipschitz}), it holds that
\begin{align}\label{Lipschitz estimate}\nonumber
	&\bold{E}\left[\int^T_0e^{\beta K_t}2|\delta_nY_t|\cdot|f(t,Y_t,\mathbb{H}_t)-f(t,Y^n_t,H^n_t)|\,dA_t\right]\\\nonumber
	&\leq\bold{E}\left[\int^T_0e^{\beta K_t}2|\delta_nY_t|\cdot(\eta_t|\delta_nY_t|+\theta_t|\delta_n\mathbb{H}_t|_{U'_t})\,dA_t\right]\\
	&\leq(2+\mu^2)\bold{E}\left[\int^T_0e^{\beta K_t}\alpha^2_t|\delta_nY_t|\,dA_t\right]+\frac{1}{\mu^2}\bold{E}\left[\int^T_0e^{\beta K_t}|\delta_n\mathbb{H}_t|^2_{U'_t}\,dA_t\right]
\end{align}
for all $\mu>0$, where we used the trivial inequality (\ref{trivial inequality}) in the second inequality.\par
For the estimate of the second term in the right hand side of the inequality (\ref{BSDEinequality}), consider the process $\langle\int^\cdot_0\delta_n\mathbb{H}_s\,d\mathbb{M}_s,\delta_nN\rangle$. Each $\mathbb{H}_{s,\omega}$ can be expanded in the covariance space $U'_{s,\omega}$ as
\begin{equation*}
	\mathbb{H}_{s,\omega}=\sum^\infty_{m=1}(\mathbb{H}_{s,\omega},e^m_{s,\omega})_{U'_{s,\omega}}e^m_{s,\omega},
\end{equation*}
where $\{e^m\}_{m\in\mathbb{N}}$ is a sequence of measure-valued predictable processes such that, for each $(s,\omega)\in[0,T]\times\Omega$, $e^m_{s,\omega}\in\text{span}\{\delta_{x_1},\dots,\delta_{x_m}\}$ for all $m\in\mathbb{N}$ and the sequence $\{e^m_{s,\omega}\}_{m\in\mathbb{N}}$ is a complete orthonormal system of $U'_{s,\omega}$; see Remark \ref{remark: CONS}. Define $\widehat{\mathbb{H}}^n$ by
\begin{equation*}
	\widehat{\mathbb{H}}^n_{s,\omega}=\sum^\infty_{m=n+1}(\mathbb{H}_{s,\omega},e^m_{s,\omega})_{U'_{s,\omega}}e^m_{s,\omega},\ (s,\omega)\in[0,T]\times\Omega.
\end{equation*}
Then obviously the process $\widehat{\mathbb{H}}^n$ is in $L^2(\mathbb{M},U')$. Furthermore, since $\mathbb{H}-\widehat{\mathbb{H}}^n$ is a linear combination of the processes $\{e^1,\dots,e^n\}$ and hence of the Dirac measures $\{\delta_{x_1},\dots,\delta_{x_n}\}$, we see that it is in $L^2(\mathbb{M}^n,\mathbb{R}^n)$. Since the martingale $\delta_nN=N-N^n$ is strongly orthogonal to $\mathbb{M}^n=(M^{x_1},\dots,M^{x_n})^{\text{tr}}$, we have that
\begin{align*}
	\left\langle\int^\cdot_0\delta_n\mathbb{H}_s\,d\mathbb{M}_s,\delta_nN\right\rangle&=\left\langle\int^\cdot_0\widehat{\mathbb{H}}^n_s\,d\mathbb{M}_s,\delta_nN\right\rangle+\left\langle\int^\cdot_0(\mathbb{H}_s-\widehat{\mathbb{H}}^n_s-H^n_s)\,d\mathbb{M}^n_s,\delta_nN\right\rangle\\
		&=\left\langle\int^\cdot_0\widehat{\mathbb{H}}^n_s\,d\mathbb{M}_s,\delta_nN\right\rangle.
\end{align*}
Hence, the second term in the right hand side of the inequality (\ref{BSDEinequality}) can be estimated as
\begin{align}\label{covariance estimate}\nonumber
	&2\bold{E}\left[\int^T_0e^{\beta K_t}\left|d\left\langle\int^\cdot_0\delta_n\mathbb{H}_s\,d\mathbb{M}_s,\delta_nN\right\rangle_t\right|\right]\\\nonumber
	&=2\bold{E}\left[\int^T_0e^{\beta K_t}\left|d\left\langle\int^\cdot_0\widehat{\mathbb{H}}^n_s\,d\mathbb{M}_s,\delta_nN\right\rangle_t\right|\right]\\\nonumber
	&\leq2\bold{E}\left[\int^T_0e^{\beta K_t}|\widehat{\mathbb{H}}^n_t|^2_{U'_t}\,dA_t\right]^{1/2}\bold{E}\left[\int^T_0e^{\beta K_t}\,d\langle\delta_nN\rangle_t\right]^{1/2}\\
	&\leq2\bold{E}\left[\int^T_0e^{\beta K_t}|\widehat{\mathbb{H}}^n_t|^2_{U'_t}\,dA_t\right]+\frac{1}{2}\bold{E}\left[\int^T_0e^{\beta K_t}\,d\langle\delta_nN\rangle_t\right],
\end{align}
where we used the Kunita--Watanabe inequality in the first inequality.\par
Combining (\ref{BSDEinequality}), (\ref{Lipschitz estimate}) and (\ref{covariance estimate}), we have that
\begin{equation}\label{BSDEinequality2}
	\left(\beta-2-\mu^2\right)\|\alpha\delta_nY\|^2_{T,\beta}+\left(1-\frac{1}{\mu^2}\right)\|\delta_n\mathbb{H}\|^2_{\mathbb{M},\beta}+\frac{1}{2}\|\delta_nN\|^2_{\mathcal{H}^2_\beta(\bold{P})}\leq2\bold{E}\left[\int^T_0e^{\beta K_t}|\widehat{\mathbb{H}}^n_t|^2_{U'_t}\,dA_t\right],
\end{equation}
for all $n\in\mathbb{N}$ and any constant $\mu>0$. Since $\beta>3$ by the assumption, we can choose a constant $\mu>0$ so that $\beta-2-\mu^2>0$ and $1-\frac{1}{\mu^2}>0$.\par
Now, since
\begin{equation*}
	e^{\beta K_t}|\widehat{\mathbb{H}}^n_t|^2_{U'_t}\leq e^{\beta K_t}|\mathbb{H}_t|^2_{U'_t}\in L^1(dA_t\otimes d\bold{P})
\end{equation*}
for all $n\in\mathbb{N}$ and
\begin{equation*}
	|\widehat{\mathbb{H}}^n_{t,\omega}|^2_{U'_{t,\omega}}=\sum^\infty_{m=n+1}|(\mathbb{H}_{t,\omega},e^m_{t,\omega})_{U'_{t,\omega}}|^2\overset{n\to\infty}{\longrightarrow}0
\end{equation*}
for all $(t,\omega)\in[0,T]\times\Omega$, the dominated convergence theorem yields that the right-hand side of the inequality (\ref{BSDEinequality2}) converges to $0$ as $n$ tends to infinity and hence we have that $\lim_{n\to\infty}\|(\delta_nY,\delta_n\mathbb{H},\delta_nN)\|_{\mathcal{S}_\beta}=0$.\par
To prove the inequality (\ref{gamma1}), consider again the second term in the right-hand side of the inequality (\ref{BSDEinequality}). By using the Kunita--Watanabe inequality and the inequality (\ref{trivial inequality}), we have that
\begin{align}\label{covariance estimate2}\nonumber
	&2\bold{E}\left[\int^T_0e^{\beta K_t}\left|d\left\langle\int^\cdot_0\delta_n\mathbb{H}_s\,d\mathbb{M}_s,\delta_nN\right\rangle_t\right|\right]\\\nonumber
	&\leq2\bold{E}\left[\int^T_0e^{\beta K_t}|\delta_n\mathbb{H}_t|^2_{U'_t}\,dA_t\right]^{1/2}\bold{E}\left[\int^T_0e^{\beta K_t}\,d\langle\delta_nN\rangle_t\right]^{1/2}\\
	&\leq\lambda^2\bold{E}\left[\int^T_0e^{\beta K_t}\,d\langle\delta_nN\rangle_t\right]+\frac{1}{\lambda^2}\bold{E}\left[\int^T_0e^{\beta K_t}|\delta_n\mathbb{H}_t|^2_{U'_t}\,dA_t\right]
\end{align}
for all $\lambda>0$. Combining (\ref{BSDEinequality}), (\ref{Lipschitz estimate}) and (\ref{covariance estimate2}), we have that
\begin{equation}
	\left(\beta-2-\mu^2\right)\|\alpha\delta_nY\|^2_{T,\beta}+\left(1-\frac{1}{\mu^2}-\frac{1}{\lambda^2}\right)\|\delta_n\mathbb{H}\|^2_{\mathbb{M},\beta}\leq(\lambda^2-1)\|\delta_nN\|^2_{\mathcal{H}^2_\beta(\bold{P})}
\end{equation}
for all $n\in\mathbb{N}$ and any constants $\mu,\lambda>0$. Since $\beta>3$, we can choose constants $\mu,\lambda>0$ so that $\beta-2-\mu^2>0$, $1-1/\mu^2-1/\lambda^2>0$ and $\lambda^2-1>0$. Then, (\ref{gamma1}) holds with $\gamma=(\lambda^2-1)\max\{(\beta-2-\mu^2)^{-1},(1-1/\mu^2-1/\lambda^2)^{-1}\}$.
\end{proof}
Suppose that the process $K=\int^\cdot_0\alpha^2_s\,dA_s$ is uniformly bounded. In this case, replacing $\alpha^2$ by $\widetilde{\alpha}^2=\alpha^2+1$ and $K$ by $\widetilde{K}=\int^\cdot_0\widetilde{\alpha}^2_s\,dA_s$ (then $\widetilde{K}$ is also bounded because of our assumption that the increasing process $A$ is bounded), we obtain the following corollary.
\begin{cor}\label{cor: BSDE}
Let $(f,\xi)$ be a data satisfying Assumption \ref{assumption: BSDEdata}. Assume that the process $K$ is uniformly bounded, $f(\cdot,0,0)\in L^2_T$, and $\xi\in L^2$. Then the following assertions hold.
\begin{enumerate}
\renewcommand{\labelenumi}{(\roman{enumi})}
\item
There exists a unique triple $(Y,\mathbb{H},N)\in L^2_T\times L^2(\mathbb{M},U')\times\mathcal{H}^2(\bold{P})$ such that the martingale $N$ is null at zero and strongly orthogonal to $\mathbb{M}=(M^x)_{x\in X}$ and that the triple $(Y,\mathbb{H},N)$ satisfies the infinite-dimensional BSDE (\ref{cylindricalBSDE}).
\item
For each $n\in\mathbb{N}$, there exists a unique triple $(Y^n,H^n,N^n)\in L^2_T\times L^2(\mathbb{M}^n,\mathbb{R}^n)\times\mathcal{H}^2(\bold{P})$ such that the martingale $N^n$ is null at zero and strongly orthogonal to $\mathbb{M}^n=(M^{x_1},\dots,M^{x_n})^{\text{tr}}$ and that the triple $(Y^n,H^n,N^n)$ satisfies the finite-dimensional BSDE (\ref{finite-dimBSDE}).
\item
For the above $(Y,\mathbb{H},N)$ and $(Y^n,H^n,N^n)$, $n\in\mathbb{N}$, it holds that
\begin{align*}
	\lim_{n\to\infty}Y^n=Y\hspace{4mm}&\text{in}\ L^2_T,\\
	\lim_{n\to\infty}H^n=\mathbb{H}\hspace{4mm}&\text{in}\ L^2(\mathbb{M},U'),\\
	\lim_{n\to\infty}N^n=N\hspace{4mm}&\text{in}\ \mathcal{H}^2(\bold{P}).
\end{align*}
\item
There exists a constant $\gamma>0$ depending only on $\|A_T\|_{L^\infty}$ and $\|K_T\|_{L^\infty}$ such that
\begin{equation}\label{gamma2}
	\|Y-Y^n\|^2_T+\|\mathbb{H}-H^n\|^2_\mathbb{M}\leq\gamma\|N-N^n\|^2_{\mathcal{H}^2(\bold{P})}
\end{equation}
for all $n\in\mathbb{N}$.
\end{enumerate}
\end{cor}
\begin{rem}
Note that the constant $\gamma>0$ appearing in (\ref{gamma1}) and (\ref{gamma2}) does not depend on the choice of the countable dense subset $\{x_1,x_2,\dots\}$ of $X$.
\end{rem}


\section{Applications to approximate hedging in bond markets}\label{section: application}
\paragraph{}
In this section, we consider a bond market modeled by the time evolution of a discounted price curve $\bar{\mathbb{P}}=((\bar{P}^T_t)_{t\in[0,T^*]})_{T\in[0,T^*]}$ of zero-coupon bonds defined on a filtered probability space $(\Omega,\mathcal{F},(\mathcal{F}_t)_{t\in[0,T^*]},\bold{P})$ satisfying the usual conditions. We assume that $T^*>0$ is finite and constant. For each $0\leq t\leq T\leq T^*$, the random variable $\bar{P}^T_t$ represents the discounted price at time $t$ of a zero-coupon bond maturing at time $T$. For $t>T$, the discounted price is usually assumed so that $\bar{P}^T_t=\bar{P}^T_T$; this assumption represents the convention that after maturity the bond is automatically transferred into the bank account. We assume that the market is frictionless and all zero-coupon bonds maturing at time $T\in[0,T^*]$ are continuously tradable.\par
Note that, in this section, the time evolution will be considered in the interval $[0,T^*]$ and each $T\in[0,T^*]$ will represent a maturity of zero-coupon bonds. With the notation of Sections \ref{section: cylindrical mart} and \ref{section: BSDE}, we take $X=[0,T^*]$ as a parameter set of a family of processes.\par
We impose the following assumptions on the market model.
\begin{assum}\label{structure condition}
\begin{itemize}\item
The discounted price curve $\bar{\mathbb{P}}=((\bar{P}^T_t)_{t\in[0,T^*]})_{T\in[0,T^*]}$ of zero-coupon bonds can be expressed as
\begin{equation*}
	\bar{\mathbb{P}}=\mathbb{M}+\int^\cdot_0b_s(\cdot)\,dA_s,
\end{equation*}
i.e., for each $T\in[0,T^*]$,
\begin{equation*}
	\bar{P}^T_t=M^T_t+\int^t_0b_s(T)\,dA_s,\ t\in[0,T^*];
\end{equation*}
\item
$\mathbb{M}=((M^T_t)_{t\in[0,T^*]})_{T\in[0,T^*]}$ is a square-integrable cylindrical martingale on $\mathcal{C}=\mathcal{C}([0,T^*])$ defined on the filtered probability space $(\Omega,\mathcal{F},(\mathcal{F}_t)_{t\in[0,T^*]},\bold{P})$ satisfying Assumptions \ref{assumption: cylindrical mart} and \ref{assumption: A is conti} with $T$ replaced by $T^*$;
\item
$b=(b_s(\cdot))_{s\in[0,T^*]}$ is a $\mathcal{C}$-valued process for which there exists a $U'$-valued predictable process $\lambda$ such that
\begin{equation}\label{MPR}
	(\mathcal{Q}_{s,\omega}\lambda_{s,\omega})(T)=b_{s,\omega}(T)
\end{equation}
for all $T\in[0,T^*]$ and $(s,\omega)\in[0,T^*]\times\Omega$ and that the nondecreasing process $\mathcal{K}_t=\int^t_0|\lambda_s|^2_{U'_s}\,dA_s$ is uniformly bounded.
\end{itemize}
\end{assum}
\begin{rem}
\begin{enumerate}
\renewcommand{\labelenumi}{(\roman{enumi})}
\item
Assumption \ref{structure condition} is a natural generalization to our infinite-dimensional setting of the so-called ``structure condition'' that is usually imposed on finite-dimensional semimartingales in mathematical finance. This condition is related to a no-arbitrage condition at least in finite-dimensional market models; see \cite{a_Delbaen-Schachermayer_95} and \cite{a_Schweizer_95}. As in the finite-dimensional case, we call the nondecreasing process $\mathcal{K}=\int^\cdot_0|\lambda_s|^2_{U'_s}\,dA_s$ the mean-variance tradeoff process.
\item
For the convention that $\bar{\mathbb{P}}^T_t=\bar{\mathbb{P}}^T_T$ for any $t>T$, we should assume that $M^T_t=M^T_T$ and $b_t(T)=0$ for all $t>T$. However, this assumption is not necessary for our results.
\end{enumerate}
\end{rem}
Note that the $\mathcal{C}$-valued process $b$ is in fact a $U$-valued predictable process because of the equality (\ref{MPR}). Furthermore, since the mean-variance tradeoff process $\mathcal{K}$ is bounded, we have that, for any $\mathbb{H}\in L^2(\mathbb{M},U')$,
\begin{align}\label{drift estimate}\nonumber
	\bold{E}\left[\left(\int^{T^*}_0|\mathbb{H}_s(b_s)|\,dA_s\right)^2\right]&=\bold{E}\left[\left(\int^{T^*}_0|(\mathbb{H}_s,\lambda_s)_{U'_s}|\,dA_s\right)^2\right]\\\nonumber
		&\leq\bold{E}\left[\left(\int^{T^*}_0|\mathbb{H}_s|_{U'_s}|\lambda_s|_{U'_s}\,dA_s\right)^2\right]\\\nonumber
		&\leq\bold{E}\left[\mathcal{K}_{T^*}\int^{T^*}_0|\mathbb{H}_s|^2_{U'_s}\,dA_s\right]\\
		&\leq\|\mathcal{K}_{T^*}\|_{L^\infty}\|\mathbb{H}\|^2_\mathbb{M}<\infty,
\end{align}
where in the first equality we use the equality (\ref{MPR}) and the fact that the map $\mathcal{Q}_{s,\omega}$ is isomorphism from $U'_{s,\omega}$ to $U_{s,\omega}$ for all $(s,\omega)\in[0,T^*]\times\Omega$, and in the third inequality we use the Cauchy--Schwarz inequality. Hence, for any $\mathbb{H}\in L^2(\mathbb{M},U')$, we can define the stochastic integral $\mathbb{H}\bullet\bar{\mathbb{P}}=\int^\cdot_0\mathbb{H}_s\,d\bar{\mathbb{P}}_s$ with respect to the infinite-dimensional semimartingale $\bar{\mathbb{P}}$ as a square-integrable semimartingale, i.e.,
\begin{equation*}
	\int^t_0\mathbb{H}_s\,d\bar{\mathbb{P}}_s=\int^t_0\mathbb{H}_s\,d\mathbb{M}_s+\int^t_0\mathbb{H}_s(b_s)\,dA_s,\ t\in[0,T^*].
\end{equation*}
See De Donno \cite{a_DeDonno_07} for more detailed discussions about generalized stochastic integrations with respect to infinite-dimensional semimartingales.\par
Let $\{T_1,T_2,\dots\}$ be a countable dense subset of the interval $[0,T^*]$. For each $n\in\mathbb{N}$, consider the $n$-dimensional processes $\bar{\mathbb{P}}^n=(\bar{P}^{T_1},\dots,\bar{P}^{T_n})^{\text{tr}}$ and $\mathbb{M}^n=(M^{T_1},\dots,M^{T_n})^{\text{tr}}$. Then for any $H\in L^2(\mathbb{M}^n,\mathbb{R}^n)$ the $n$-dimensional stochastic integral $H\bullet\bar{\mathbb{P}}^n=\int^\cdot_0H_s\,d\bar{\mathbb{P}}^n_s$ is well-defined and is a square-integrable semimartingale with the canonical decomposition
\begin{equation*}
	\int^t_0H_s\,d\bar{\mathbb{P}}^n_s=\int^t_0H_s\,d\mathbb{M}^n_s+\sum^n_{i=1}\int^t_0H^i_sb_s(T_i)\,dA_s.
\end{equation*}
We often identify the $\mathbb{R}^n$-valued process $H=(H^1,\dots,H^n)^{\text{tr}}$ with the measure-valued process $H=\sum^n_{i=1}H^i\delta_{T^i}$ as in Section \ref{section: BSDE}. We call the $n$-dimensional semimartingale $\bar{\mathbb{P}}^n$ the $n$-th small market.\par
Let us introduce the notion of trading strategies. In small markets, we adopt the definition of (finite-dimensional) strategies in Schweizer \cite{ic_Schweizer_01,ic_Schweizer_08}.
\begin{defi}\label{def: strategy in small markets}
Fix $n\in\mathbb{N}$. An $L^2$-strategy in the $n$-th small market is a pair $\phi=(H,\eta)$, where $H\in L^2(\mathbb{M}^n,\mathbb{R}^n)$ and $\eta$ is an real-valued adapted process such that the (discounted) value process
\begin{equation*}
	V_t(\phi)=\eta_t+H_t(\bar{\mathbb{P}}^n_t)=\eta_t+\sum^n_{i=1}H^i_t\bar{P}^{T_i}_t
\end{equation*}
is right-continuous and square-integrable, i.e., $V_t(\phi)\in L^2(\bold{P})$ for all $t\in[0,T^*]$.\par
Let $\phi=(H,\eta)$ be an $L^2$-strategy in the $n$-th small market. The (cumulative) cost process of $\phi$ is
\begin{equation*}
	C_t(\phi)=V_t(\phi)-\int^t_0H_s\,d\bar{\mathbb{P}}^n_s.
\end{equation*}
$\phi$ is called self-financing if the cost process satisfies $C_t(\phi)\equiv c$ for some $\mathcal{F}_0$-measurable random variable $c\in L^2(\mathcal{F}_0,\bold{P})$, and mean-self-financing if $C(\phi)$ is a martingale (which is then square-integrable).\par
The risk process of $\phi$ is
\begin{equation*}
	R_t(\phi)=\bold{E}\left[(C_{T^*}(\phi)-C_t(\phi))^2\relmiddle|\mathcal{F}_t\right],\ t\in[0,T^*].
\end{equation*}
We set $R(\phi)=\bold{E}\left[(C_{T^*}(\phi)-C_0(\phi))^2\right]$ and call it the total risk of $\phi$.
\end{defi}
For an $L^2$-strategy $\phi=(H,\eta)$ in a small market, the finite-dimensional integrand $H$ represents the amount of holdings of zero-coupon bonds with fixed (finite number of) maturities, and the real adapted process $\eta$ represents the position in the bank account. Then the stochastic integral $\int^\cdot_0H_s\,d\bar{\mathbb{P}}^n_s$ is the cumulative gains corresponding to the investment.
\begin{prop}\label{prop: finite-dimensional strategy}
\begin{enumerate}
\renewcommand{\labelenumi}{(\roman{enumi})}
\item
For given right-continuous, adapted, square-integrable process $C$ and $H\in L^2(\mathbb{M}^n,\mathbb{R}^n)$, there exists a unique real-valued adapted process $\eta$ such that $\phi=(H,\eta)$ is an $L^2$-strategy in the $n$-th small market having the cost process $C$.
\item
For given $\xi\in L^2(\mathcal{F}_{T^*},\bold{P})$ and $H\in L^2(\mathbb{M}^n,\mathbb{R}^n)$, there exist a unique square-integrable martingale $C$ and a unique real-valued adapted process $\eta$ such that $\phi=(H,\eta)$ is an $L^2$-strategy in the $n$-th small market having the cost process $C$ and satisfying $V_{T^*}(\phi)=\xi$ $\bold{P}$-a.s.
\end{enumerate}
Here, the term ``unique'' means ``unique up to indistinguishability''.
\end{prop}
\begin{proof}
(\rnum{1}): Define $\phi=(H,\eta)$ via
\begin{equation*}
	\eta_t=\int^t_0H_s\,d\bar{\mathbb{P}}^n_s-H_t(\bar{\mathbb{P}}^n_t)+C_t,\ t\in[0,T^*].
\end{equation*}
Then the value process can be written as $V(\phi)=\int^\cdot_0H_s\,d\bar{\mathbb{P}}^n_s+C$. By the assumptions the process $V(\phi)$ is right-continuous and square-integrable and hence $\phi$ is an $L^2$-strategy in the $n$-th small market. The cost process $C(\phi)$ of $\phi$ coincides with $C$. Let $\phi'=(H,\eta')$ be an $L^2$-strategy in the $n$-th small market with the cost process $C$. Then the right-continuity of the value processes $V(\phi)$ and $V(\phi')$ yields that these two processes are indistinguishable and hence $\eta$ and $\eta'$ are also indistinguishable.\\
(\rnum{2}): Set
\begin{equation*}
	\widetilde{C}_t=\bold{E}\left[\xi-\int^{T^*}_0H_s\,d\bar{\mathbb{P}}^n_s\relmiddle|\mathcal{F}_t\right],\ t\in[0,T^*],
\end{equation*}
and let $C$ be the right-continuous version of $\widetilde{C}$. Then the process $C$ is a right-continuous square-integrable martingale. For the pair $(C,H)$, define a real-valued adapted process $\eta$ as in (\rnum{1}). Then $\phi=(H,\eta)$ is an $L^2$-strategy in the $n$-th small market with the cost process $C$ and satisfies $V_{T^*}(\phi)=\xi$. Let $(C',\eta')$ be a pair satisfying the assertion in (\rnum{2}). Then the martingale property of $C'$ yields that
\begin{equation*}
	C'_t=\mathbb{E}\left[C'_{T^*}\relmiddle|\mathcal{F}_t\right]=\bold{E}\left[V_{T^*}(\phi')-\int^{T^*}_0H_s\,d\bar{\mathbb{P}}^n_s\relmiddle|\mathcal{F}_t\right]=\bold{E}\left[\xi-\int^{T^*}_0H_s\,d\bar{\mathbb{P}}^n_s\relmiddle|\mathcal{F}_t\right]=C_t
\end{equation*}
$\bold{P}$-a.s. for all $t\in[0,T^*]$ and then the right-continuity of $C$ and $C'$ yields that these two martingales are indistinguishable. Now the indistinguishability of the processes $\eta$ and $\eta'$ follows from (\rnum{1}).
\end{proof}
Hence, in small markets, an $L^2$-strategy can be characterized by a finite-dimensional integrand $H$ and a cost process $C$. Moreover, For any claim $\xi\in L^2(\mathcal{F}_{T^*},\bold{P})$, there exists a mean-self-financing $L^2$-strategy in small markets which achieves the claim $\xi$ with a cost process $C$, and this martingale $C$ is completely determined by the finite-dimensional integrand $H$ (and of course by the claim $\xi$).\par
With this fact in mind, let us define a generalized strategy in the whole market.
\begin{defi}
We call a pair $\Phi=(\mathbb{H},C)$ a generalized $L^2$-strategy if $\mathbb{H}\in L^2(\mathbb{M},U')$ and if $C$ is a real-valued, right-continuous, square-integrable process. The generalized value process corresponding to $\Phi$ is defined by
\begin{equation}
	V_t(\Phi)=\int^t_0\mathbb{H}_s\,d\bar{\mathbb{P}}_s+C_t,\ t\in[0,T^*].
\end{equation}
$\Phi$ is called self-financing if the cost process satisfies $C_t(\phi)\equiv c$ for some $\mathcal{F}_0$-measurable random variable $c\in L^2(\mathcal{F}_0,\bold{P})$, and mean-self-financing if $C$ is a (square-integrable) martingale.
\end{defi}
\begin{rem}
If $\mathbb{H}$ is in $L^2(\mathbb{M},\mathcal{M})$, then it can be regarded as the holdings of zero-coupon bonds with possibly infinitely many maturities. In general, however, a generalized $L^2$-strategy $\Phi=(\mathbb{H},C)$ does not represent a portfolio of the trader in the classical sense. Since each $\bar{\mathbb{P}}_t$ is not necessarily in $U_t$, the position ``$\mathbb{H}_t(\bar{\mathbb{P}}_t)$'' in zero-coupon bonds is not defined in general, and hence the position ``$\eta$'' in the bank account cannot be identified by the generalized strategy $\Phi$. A generalized strategy is fictitious but can be approximated by a sequence of realistic strategies in small markets, and so we should take its reasonable approximate sequence. 
\end{rem}
Before constructing reasonable approximations for generalized strategies, let us introduce the notion of approximate attainability of a claim.
\begin{defi}
A contingent claim is a random variable $\xi\in L^2(\mathcal{F}_{T^*},\bold{P})$. The claim $\xi$ is said to be approximately attainable if there exists a self-financing generalized $L^2$-strategy $\Phi=(\mathbb{H},c)$ such that
\begin{equation*}
	\xi=V_{T^*}(\Phi)=c+\int^{T^*}_0\mathbb{H}_s\,d\bar{\mathbb{P}}_s\ \bold{P}\text{-a.s.}
\end{equation*}
The market $\bar{\mathbb{P}}$ is said to be approximately complete if all contingent claims are approximately attainable.
\end{defi}
\begin{rem}
It is worth noting that our definitions of the set of contingent claims and approximate completeness are based on the subjective measure $\bold{P}$.
\end{rem}
Proposition \ref{approximately attainable} below characterizes the approximate attainability of a claim in terms of the solution of the corresponding BSDE. Set $f(\omega,t,y,\mathbb{H}_{t,\omega})=-\mathbb{H}_{t,\omega}(b_{t,\omega})$, $\alpha^2_t=1+|\lambda_t|^2_{U'_t}$, and $K_t=\int^t_0\alpha^2_s\,dA_s=\mathcal{K}_t+A_t$ with the notations in Section \ref{section: BSDE} and recall the assumptions that the mean-variance tradeoff process $\mathcal{K}=\int^\cdot_0|\lambda_t|^2_{U'_t}\,dA_t$ is bounded and that the increasing process $A$ is continuous and bounded. Then, for any claim $\xi\in L^2(\mathcal{F}_{T^*},\bold{P})$, the data $(f,\xi)$ satisfies all the assumptions in Corollary \ref{cor: BSDE}. Hence, there exists a unique solution $(Y,\mathbb{H},N)\in L^2_{T^*}\times L^2(\mathbb{M},U')\times\mathcal{H}^2(\bold{P})$ of the BSDE
\begin{equation}\label{hedgeBSDE}
	Y_t=\xi-\int^{T^*}_t\mathbb{H}_s(b_s)\,dA_s-\int^{T^*}_t\mathbb{H}_s\,d\mathbb{M}_s-\int^{T^*}_tdN_s,\ t\in[0,T^*].
\end{equation}
\begin{prop}\label{approximately attainable}
For a claim $\xi\in L^2(\mathcal{F}_{T^*},\bold{P})$, the following conditions are equivalent.
\begin{enumerate}
\renewcommand{\labelenumi}{(\roman{enumi})}
\item
$\xi$ is approximately attainable.
\item
The unique solution $(Y,\mathbb{H},N)\in L^2_{T^*}\times L^2(\mathbb{M},U')\times\mathcal{H}^2(\bold{P})$ of $\text{BSDE}(f,\xi)$ with $f(\omega,t,y,\mathbb{H}_{t,\omega})=-\mathbb{H}_{t,\omega}(b_{t,\omega})$ satisfies $N\equiv0$.
\end{enumerate}
Moreover, in this case, the self-financing generalized $L^2$-strategy which attains the claim $\xi$ is $\Phi=(\mathbb{H},Y_0)$ and the corresponding value process is $V(\Phi)=Y$.
\end{prop}
\begin{proof}
Suppose that the claim $\xi\in L^2(\mathcal{F}_{T^*},\bold{P})$ is approximately attainable. Then there exists a self-financing generalized $L^2$-strategy $\Phi=(\mathbb{H},c)$ such that $\xi=V_{T^*}(\Phi)$. Now let $Y=V(\Phi)$. By the definition of the value process, $Y$ is right-continuous, adapted, and square-integrable. Moreover, the process $Y$ is in $L^2_{T^*}$. Indeed, we have that
\begin{equation*}
	\|Y\|^2_{T^*}=\bold{E}\left[\int^{T^*}_0Y^2_t\,dA_t\right]\leq T^*\|A_{T^*}\|_{L^\infty}\bold{E}\left[\sup_{t\in[0,T^*]}Y^2_t\right].
\end{equation*}
For each $t\in[0,T^*]$, we have
\begin{align*}
	Y^2_t&\leq2c^2+2\left(\int^t_0\mathbb{H}_s\,d\bar{\mathbb{P}}_s\right)^2\\
		&\leq2c^2+4\left(\int^t_0\mathbb{H}_s\,d\mathbb{M}_s\right)^2+4\left(\int^t_0\mathbb{H}_s(b_s)\,dA_s\right)^2\\
		&\leq2c^2+4\left(\int^t_0\mathbb{H}_s\,d\mathbb{M}_s\right)^2+4\left(\int^{T^*}_0|\mathbb{H}_s(b_s)|\,dA_s\right)^2.
\end{align*}
Since $c\in L^2(\mathcal{F}_0,\bold{P})$, $\mathbb{H}\in L^2(\mathbb{M},U')$ and the increasing process $A$ is bounded, by using the Burkholder--Davis--Gundy inequality and the estimate (\ref{drift estimate}), we see that $\|Y\|_{T^*}<\infty$ and hence $Y\in L^2_{T^*}$. Since
\begin{equation*}
	\xi=V_{T^*}(\Phi)=Y_{T^*}=Y_t+\int^{T^*}_t\mathbb{H}_s\,d\bar{\mathbb{P}}_s=Y_t+\int^{T^*}_t\mathbb{H}_s(b_s)\,dA_s+\int^{T^*}_t\mathbb{H}_s\,d\mathbb{M}_s
\end{equation*}
for each $t\in[0,T^*]$, we see that the triple $(Y,\mathbb{H},0)\in L^2_{T^*}\times L^2(\mathbb{M},U')\times\mathbb{H}^2(\bold{P})$ is the (unique) solution of the $\text{BSDE}(f,\xi)$ with $f(\omega,t,y,\mathbb{H}_{t,\omega})=-\mathbb{H}_{t,\omega}(b_{t,\omega})$.\par
Conversely, assume that the assertion (\rnum{2}) holds. Then
\begin{equation*}
Y_t=\xi-\int^{T^*}_t\mathbb{H}_s(b_s)\,dA_s-\int^{T^*}_t\mathbb{H}_s\,d\mathbb{M}_s=\xi-\int^{T^*}_t\mathbb{H}_s\,d\bar{\mathbb{P}}_s,\ t\in[0,T^*],
\end{equation*}
in particular, $\xi=Y_0+\int^{T^*}_0\mathbb{H}_s\,d\bar{\mathbb{P}}_s$ $\bold{P}$-a.s. Since the claim $\xi$ and the stochastic integral $\int^{T^*}_0\mathbb{H}_s\,d\bar{\mathbb{P}}_s$ is in $L^2(\bold{P})$, the random variable $Y_0$ is in $L^2(\mathcal{F}_0,\bold{P})$. Hence, the pair $\Phi=(\mathbb{H},Y_0)$ is a self-financing generalized $L^2$-strategy and satisfies $V_{T^*}(\Phi)=\xi$ $\bold{P}$-a.s. The remaining assertions are trivial.
\end{proof}
For an approximately attainable claim $\xi$, there exists a self-financing generalized $L^2$-strategy $\Phi$ which formally replicates the claim. However, this strategy $\Phi$ is fictitious and we should construct a reasonable approximate sequence for $\Phi$ which is meaningful in finance.\par
To this end, we focus on the notions of locally risk-minimizing strategies and the F\"{o}llmer--Schweizer decomposition of contingent claims in small markets that were introduced by Schweizer \cite{ic_Schweizer_01,ic_Schweizer_08}. We shall recall the definitions and the corresponding results in our model.
\begin{defi}
\begin{enumerate}
\renewcommand{\labelenumi}{(\roman{enumi})}
Fix $n\in\mathbb{N}$ and consider the $n$-th small market $\bar{\mathbb{P}}^n$.
\item
A pair $\Delta=(\delta,\epsilon)$ consisting of an $\mathbb{R}^n$-valued predictable process $\delta$ and a real-valued adapted process $\epsilon$ is called a small perturbation in the $n$-th small market if $\langle\int^\cdot_0\delta_s\,d\mathbb{M}^n_s\rangle=\int^\cdot_0|\delta_s|^2_{U'_s}\,dA_s$ (and hence also $\int^\cdot_0|\sum^n_{i=1}\delta^i_sb_s(T_i)|\,dA_s$) are uniformly bounded, the pair $V(\Delta)=\epsilon+\sum^n_{i=1}\delta^i\bar{P}^{T_i}$ is square-integrable, and $V_{T^*}(\Delta)=0$ $\bold{P}$-a.s.
\item
For each small perturbation $\Delta=(\delta,\epsilon)$ and subinterval $(s,t]$ of $[0,T^*]$, we define the small perturbation
\begin{equation*}
	\Delta|_{(s,t]}=
\begin{cases}
		(\delta\1_{\rsi s,t\rsi},\epsilon\1_{\lsi s,t\lsi})\ &\text{if}\ t<T^*,\\
		(\delta\1_{\rsi s,t\rsi},\epsilon\1_{\lsi s,T^*\rsi})\ &\text{if}\ t=T^*.
\end{cases}
\end{equation*}
For an $L^2$ strategy $\phi$ and a small perturbation $\Delta$, both in the $n$-th small market, and a partition $\tau=\{t_0,t_1,\dots,t_k\}\subset[0,T^*]$ with $0=t_0<t_1<\dots<t_k=T^*$, we set
\begin{equation*}
	r^\tau[\phi,\Delta]=\sum^{k-1}_{i=0}\frac{R_{t_i}(\phi+\Delta|_{(t_i,t_{i+1}]})-R_{t_i}(\phi)}{\bold{E}[A_{t_{i+1}}-A_{t_i}\ |\ \mathcal{F}_{t_i}]}\1_{\rsi t_i,t_{i+1}\rsi}.
\end{equation*}
(Note that the process $A$ is strictly increasing and bounded by our assumptions.)
\item
An $L^2$-strategy $\phi$ in the $n$-th small market is called locally risk-minimizing if for every small perturbation $\Delta$ in the $n$-th small market and every increasing sequence $(\tau_k)_{k\in\mathbb{N}}$ of partitions tending to the identity, we have
\begin{equation*}
	\liminf_{k\to\infty}r^{\tau_k}[\phi,\Delta]\geq0\ dA_t\otimes d\bold{P}\text{-a.s.}
\end{equation*}
\end{enumerate}
\end{defi}
In our model, since the process $A$ and hence the mean-variance tradeoff process $\mathcal{K}$ is continuous, locally risk-minimizing strategies in small markets are characterized by the following theorem.
\begin{theo}[Schweizer \cite{ic_Schweizer_08}]\label{theo: Schweizer}
Suppose that a claim $\xi\in L^2(\mathcal{F}_{T^*},\bold{P})$ is given. Fix $n\in\mathbb{N}$. Then, for an $L^2$-strategy $\phi$ in the $n$-th small market satisfying $V_{T^*}(\phi)=\xi$ $\bold{P}$-a.s., the following assertions are equivalent.
\begin{enumerate}
\renewcommand{\labelenumi}{(\roman{enumi})}
\item
$\phi$ is locally risk-minimizing in the $n$-th small market.
\item
$\phi$ is mean-self-financing and the cost function $C(\phi)$ is strongly orthogonal to $\mathbb{M}^n$.
\end{enumerate}
\end{theo}
Because of Theorem \ref{theo: Schweizer}, locally risk-minimizing strategies are related to the F\"{o}llmer--Schweizer decomposition of a claim defined below.
\begin{defi}
For each $n\in\mathbb{N}$, we say that a claim $\xi\in L^2(\mathcal{F}_{T^*},\bold{P})$ admits the F\"{o}llmer--Schweizer decomposition in the $n$-th small market if it is expressed as
\begin{equation*}
	\xi=\xi^{0,n}+\int^{T^*}_0H^{\xi,n}_s\,d\bar{\mathbb{P}}^n_s+N^{\xi,n}_{T^*}\ \bold{P}\text{-a.s.},
\end{equation*}
where $\xi^{0,n}\in L^2(\mathcal{F}_0,\bold{P})$, $H^{\xi,n}\in L^2(\mathbb{M}^n,\mathbb{R}^n)$, and the process $N^{\xi,n}=(N^{\xi,n}_t)_{t\in[0,T^*]}$ is a (right-continuous) square-integrable martingale that is null at zero and strongly orthogonal to $\mathbb{M}^n$.
\end{defi}
Fix any claim $\xi\in L^2(\mathcal{F}_{T^*},\bold{P})$ and $n\in\mathbb{N}$. Consider the finite-dimensional $\text{BSDE}(f,\xi)$ with $f(\omega,t,y,h)=-h(b_{t,\omega})=-\sum^n_{i=1}h^ib_{t,\omega}(T_i)$ for $h\in\mathbb{R}^n$;
\begin{equation}\label{Follmer-Schweizer BSDE}
	Y^n_t=\xi-\int^{T^*}_tH^n_s(b_s)\,dA_s-\int^{T^*}_tH^n_s\,d\mathbb{M}^n_s-\int^{T^*}_tdN^n_s,\ t\in[0,T^*].
\end{equation}
By Corollary \ref{cor: BSDE}, there exists a unique solution $(Y^n,H^n,N^n)\in L^2_{T^*}\times L^2(\mathbb{M}^n,\mathbb{R}^n)\times\mathcal{H}^2(\bold{P})$ of the BSDE (\ref{Follmer-Schweizer BSDE}). In particular, the claim $\xi$ can be written as
\begin{align*}
	\xi&=Y^n_0+\int^{T^*}_0H^n_s(b_s)\,dA_s+\int^{T^*}_0H^n_s\,d\mathbb{M}^n_s+N^n_{T^*}\\
		&=Y^n_0+\int^{T^*}_0H^n_s\,d\bar{\mathbb{P}}^n_s+N^n_{T^*}\ \bold{P}\text{-a.s.}
\end{align*}
Since the square-integrable martingale $N^n$ is null at zero and strongly orthogonal to $\mathbb{M}^n$, the claim $\xi$ admits the unique F\"{o}llmer--Schweizer decomposition in the $n$-th small market.\par
Define a right-continuous square-integrable martingale $C^n$ by $C^n_t=Y^n_0+N^n_t$, $t\in[0,T^*]$, and consider the corresponding process $\eta^n$ as in Proposition \ref{prop: finite-dimensional strategy} with respect to the finite-dimensional integrand $H^n\in L^2(\mathbb{M}^n,\mathbb{R}^n)$ and the martingale $C^n$. Then the pair $\phi^n=(H^n,\eta^n)$ is a mean-self-financing $L^2$-strategy in the $n$-th small market with the cost process $C^n$ satisfying $V_{T^*}(\phi^n)=\xi$ $\bold{P}$-a.s. Since $C^n$ is a square-integrable martingale strongly orthogonal to $\mathbb{M}^n$, by Theorem \ref{theo: Schweizer}, $\phi^n$ is locally risk-minimizing in the $n$-th small market.\par
Now we can state our result about a reasonable approximate sequence of a generalized strategy.
\begin{theo}\label{theorem: application}
Let $\xi\in L^2(\mathcal{F}_{T^*},\bold{P})$ be an approximately attainable claim, $\Phi=(\mathbb{H},c)$ be the self-financing generalized $L^2$-strategy satisfying $V_{T^*}(\Phi)=\xi$ $\bold{P}$-a.s., and $\phi^n=(H^n,\eta^n)$ be the locally risk-minimizing $L^2$-strategy in the $n$-th small market $\bar{\mathbb{P}}^n$ satisfying $V_{T^*}(\phi^n)=\xi$, for each $n\in\mathbb{N}$. Then it holds that
\begin{align*}
	&\lim_{n\to\infty}V(\phi^n)=V(\Phi)\hspace{4mm}\text{in}\ L^2_{T^*},\\
	&\lim_{n\to\infty}H^n=\mathbb{H}\hspace{4mm}\text{in}\ L^2(\mathbb{M},U'),\\
	&\lim_{n\to\infty}R(\phi^n)=0,
\end{align*}
where each $R(\phi^n)$ is the total risk of $\phi^n$ defined in Definition \ref{def: strategy in small markets}. Moreover, there exists a constant $\gamma>0$ depending only on the mean-variance tradeoff process $\mathcal{K}$ and the increasing process $A$ such that
\begin{equation*}
	\|V(\Phi)-V(\phi^n)\|^2_{T^*}+\|\mathbb{H}-H^n\|^2_\mathbb{M}\leq\gamma R(\phi^n)
\end{equation*}
for all $n\in\mathbb{N}$.
\end{theo}
\begin{proof}
By Proposition \ref{approximately attainable}, $(V(\Phi),\mathbb{H},0)\in L^2_{T^*}\times L^2(\mathbb{M},U')\times\mathcal{H}^2(\bold{P})$ is the unique solution of the infinite-dimensional BSDE (\ref{hedgeBSDE}). On the other hand, for each $n\in\mathbb{N}$, $(V(\phi^n),H^n,N^n)\in L^2_{T^*}\times L^2(\mathbb{M}^n,\mathbb{R}^n)\times\mathcal{H}^2(\bold{P})$ with $N^n=C(\phi^n)-V_0(\phi^n)=C(\phi^n)-C_0(\phi^n)$ is the unique solution of the finite-dimensional BSDE (\ref{Follmer-Schweizer BSDE}). Note that, for each $n\in\mathbb{N}$, we have
\begin{equation}
	\|N^n\|^2_{\mathcal{H}^2(\bold{P})}=\|N^n_{T^*}\|^2_{L^2(\bold{P})}=\bold{E}\left[(C_{T^*}(\phi^n)-C_0(\phi^n))^2\right]=R(\phi^n).
\end{equation}
Hence, the assertions hold by Corollary \ref{cor: BSDE}.
\end{proof}
\begin{rem}
As mentioned in Schweizer \cite{ic_Schweizer_01}, for the locally risk-minimizing strategy $\phi^n$ in the $n$-th small market, the process $V(\phi^n)$ can be interpreted as an intrinsic valuation process for the claim $\xi$ based on the $n$-th small market with respect to the subjective criterion of local risk-minimization. Theorem \ref{theorem: application} says that, for an approximately attainable claim, the sequence of intrinsic valuation processes and that of optimal strategies based on small markets converge to the generalized valuation process and the generalized hedging strategy respectively in $L^2$-sense. Furthermore, the corresponding total risks tend to zero. In such a sense, we can construct a ``reasonable'' approximate sequence for a generalized strategy in a bond market. Note that, as shown in De Donno, Guasoni, and Pratelli \cite{a_DeDonno-_05}, in an incomplete large financial market, the super-replication prices of a contingent claim based on small markets do not necessarily converge to the one based on the large market. On the other hand, Theorem \ref{theorem: application} justifies the approximation procedure of valuation processes in terms of risks and gives a new perspective to the theory of pricing contingent claims in infinite-dimensional market models.
\end{rem}


\section*{Acknowledgments}
\paragraph{}
I would like to thank Professor Masanori Hino, who is my supervisor, and Professor Ichiro Shigekawa, for helpful discussions. I also wish to express my gratitude to Professor Jun Sekine for giving me a lot of valuable comments regarding researches on BSDEs and bond markets.


\bibliography{reference}


\end{document}